\documentclass[12pt]{amsart}
\usepackage{amsmath,amsthm,amsfonts,amscd,amssymb,eucal,latexsym,mathrsfs,stmaryrd,enumitem}
\usepackage[all]{xy}

\newtheorem{theorem}{Theorem}[section]
\newtheorem{corollary}[theorem]{Corollary}
\newtheorem{lemma}[theorem]{Lemma}
\newtheorem{proposition}[theorem]{Proposition}

\theoremstyle{definition}
\newtheorem{definition}[theorem]{Definition}
\newtheorem{remark}[theorem]{Remark}

\newtheorem{example}[theorem]{Example}

\newtheorem{problem}[theorem]{Problem}

\newcommand{\cC}{{\mathcal C}}
\newcommand{\cU}{{\mathcal U}}

\newcommand{\cT}{{\mathcal T}}

\newcommand{\cP}{{\mathcal P}}

\newcommand{\cZ}{{\mathcal Z}}

\newcommand{\cV}{{\mathcal V}}

\newcommand{\sC}{{\mathscr C}}

\newcommand{\sP}{{\mathscr P}}

\newcommand{\Cb}{{\mathbb C}}
\newcommand{\Zb}{{\mathbb Z}}
\newcommand{\Tb}{{\mathbb T}}

\newcommand{\Nb}{{\mathbb N}}
\newcommand{\Zbn}{{\mathbb Z}_{\geq 0}}

\newcommand{\eps}{\varepsilon}
\newcommand{\unit}{1}

\DeclareMathOperator{\diam}{diam}

\DeclareMathOperator{\towdim}{dim_{tow}}

\DeclareMathOperator{\nucdim}{dim_{nuc}}
\DeclareMathOperator{\nucdimplus}{dim_{nuc}^{+1}}
\DeclareMathOperator{\amdim}{dim_{am}}

\DeclareMathOperator{\asdim}{asdim}

\DeclareMathOperator{\ftowdim}{dim_{ftow}}
\DeclareMathOperator{\dad}{dad}
\DeclareMathOperator{\dadplus}{dad^{+1}}
\DeclareMathOperator{\ftowdimplus}{dim_{ftow}^{+1}}

\DeclareMathOperator{\dimplus}{dim^{+1}}

\DeclareMathOperator{\towdimplus}{dim_{tow}^{+1}}
\DeclareMathOperator{\amdimplus}{dim_{am}^{+1}}

\allowdisplaybreaks

\begin{document}

\title{Dimension, comparison, and almost finiteness}
\author{David Kerr}
\address{
Department of Mathematics,
Texas A\&M University, 
College Station, TX 77843-3368, USA}
\email{kerr@math.tamu.edu}

\date{October 30, 2018}

\begin{abstract}
We develop a dynamical version of some of the theory 
surrounding the Toms--Winter conjecture for simple separable nuclear C$^*$-algebras
and study its connections to the C$^*$-algebra side via the crossed product.
We introduce an analogue of hyperfiniteness for free actions of amenable groups 
on compact spaces and show that it plays the role of $\cZ$-stability in the 
Toms--Winter conjecture in its relation to dynamical comparison, and 
also that it implies $\cZ$-stability of the crossed product. This property, 
which we call almost finiteness,
generalizes Matui's notion of the same name from the zero-dimensional setting.
We also introduce a notion of tower dimension as a partial analogue of nuclear dimension
and study its relation to dynamical comparison and almost finiteness,
as well as to the dynamical asymptotic dimension and amenability dimension
of Guentner, Willett, and Yu.
\end{abstract}

\maketitle

\section{Introduction}

Two of the cornerstones of the theory of von Neumann algebras with separable predual are the following
theorems due to Murray--von Neumann \cite{MurvNe43} and Connes \cite{Con76}, respectively:
\begin{enumerate}
\item there is a unique hyperfinite II$_1$ factor,

\item injectivity is equivalent to hyperfiniteness.
\end{enumerate}
Injectivity is a form of amenability that gives operator-algebraic expression to the 
idea of having an invariant mean, while hyperfiniteness means that the algebra
can be expressed as the weak operator closure
of an increasing sequence of finite-dimensional $^*$-subalgebras
(or, equivalently, that one has local $^*$-ultrastrong approximation by such $^*$-subalgebras \cite{EllWoo76}).
The basic prototype for the relation between an invariant-mean-type property
and finite or finite-dimensional approximation
is the equivalence between amenability and the F{\o}lner property
for discrete groups, and indeed Connes's proof of (ii)
draws part of its inspiration from the Day--Namioka proof of this equivalence.

In the theory of measured equivalence relations on standard probability spaces one has 
the following analogous pair of results, the first of which is a theorem of Dye \cite{Dye59}
and the second of which is the Connes--Feldmann--Weiss theorem \cite{ConFelWei81}
(here p.m.p.\ stands for probability-measure-preserving):
\begin{enumerate}
\item[(iii)] there is a unique hyperfinite ergodic p.m.p.\ equivalence relation,

\item[(iv)] amenability is equivalent to hyperfiniteness.
\end{enumerate}
Again amenability is defined as the existence of a suitable type of invariant mean, while hyperfiniteness
means that the relation is equal a.e.\ to an increasing union of subrelations with finite classes.
Thus among both
II$_1$ factors and p.m.p.\ equivalence relations there is a unique amenable object 
and it can be characterized via a finite or finite-dimensional approximation property.
The two settings are furthermore linked in a direct technical way by the
equivalence of the following three conditions for a free p.m.p.\ action of a countably infinite group:
\begin{enumerate}
\item[(v)] the orbit equivalence relation of the action is hyperfinite,

\item[(vi)] the crossed product is isomorphic to the unique hyperfinite II$_1$ factor,

\item[(vii)] the group is amenable.
\end{enumerate}
The implication (vii)$\Rightarrow$(v) was established
by Ornstein and Weiss as a consequence of their Rokhlin-type tower theorem \cite{OrnWei80,OrnWei87}
and can also be deduced from the Connes--Feldman--Weiss theorem. The implication
(vii)$\Rightarrow$(vi) was established by
Connes as an application of his result that injectivity implies hyperfiniteness and can 
also be derived in a more elementary way using the Ornstein--Weiss tower theorem
(it is interesting to note however that one needs the full force of Connes's theorem
in order to show that the group von Neumann algebra of an amenable group is hyperfinite).

In the type III case there is a similarly definitive theory, with the isomorphism classes 
being much more abundant but still classifiable in a nice way. For the present discussion
however we will leave this aside, since our focus will be on 
amenable type II phenomena in the topological-dynamical and C$^*$-algebraic realms, 
where the uniqueness in (i) and (iii) already gets replaced by a vast 
array of possible behaviour for which a complete classification is likely hopeless
without the addition of further regularity hypotheses.
In fact our principal aim has been to clarify what kind of regularity
properties on the dynamical side match up, at least through analogy and one-way implications,
with the key regularity properties of finite nuclear dimension, $\cZ$-stability,
and strict comparison that have helped set the stage for the dramatic advances
made over the last few years in the classification program for simple separable 
nuclear (i.e., amenable) C$^*$-algebras. In the process we will try to reimagine
the equivalence of (v) and (vi) in the context of actions on compact metrizable spaces
by introducing an analogue of hyperfiniteness and relating it to the $\cZ$-stability
of the crossed product.

For C$^*$-algebras, the strictest and simplest technical analogue of a hyperfinite 
von Neumann algebra would be an AF algebra, which similarly means that the algebra 
can be expressed as the closure of an increasing union of 
finite-dimensional $^*$-subalgebras (or, equivalently, 
that one has local approximation by such $^*$-subalgebras), 
but with the weak operator topology replaced by the norm topology.
In the 1970s, separable AF algebras were shown to be classified by their ordered $K$-theory (Elliott)
as well as by related combinatorial objects called Bratteli diagrams (Bratteli).
This reinforced the affinity with von-Neumann-algebraic hyperfiniteness by revealing a 
parallel structural tractability, however different the nature of the invariants.

What is remarkable is that the classification of AF algebras 
ended up being only the beginning of a much more ambitious
program that was launched in the 1980s by Elliott, who realized that $C^*$-inductive
limits of more general types of building blocks could be classified
by ordered $K$-theory paired with traces and suggested that a similar classification might hold for even
larger classes of (or perhaps even all) separable nuclear C$^*$-algebras.
The Elliott program has experienced
many successes and several surprising twists over the last twenty-five years 
through the efforts of many researchers 
and has recently culminated, in the simple unital UCT case, with a definitive classification 
which merely assumes the abstract regularity hypothesis of finite nuclear dimension
(this result combines theorems of Gong--Lin--Niu \cite{GonLinNiu15}, 
Elliott--Gong--Lin--Niu \cite{EllGonLinNiu15}, and White--Winter--Tikuisis \cite{TikWhiWin17},
while also incorporating the earlier Kirchberg--Phillips classification 
on the purely infinite side \cite{Kir94,Phi00}).
The UCT (universal coefficient theorem) is a homological condition relating
$K$-theory and $KK$-theory which is possibly redundant and is 
automatic for crossed products of actions of countable amenable groups 
on compact metrizable spaces by a result of Tu \cite{Tu99}.

That classifiability of simple separable unital C$^*$-algebras
in the UCT class now boils down to the simple question of whether the nuclear dimension is finite
belies the critical role that several other regularity properties have played and continue to play
in classification theory. The most important among these
are $\cZ$-stability (i.e., tensorial absorption of 
the Jiang--Su algebra $\cZ$), strict comparison, and tracial rank conditions.
Strict comparison is a C$^*$-algebraic version of the property 
that the comparability of projections in a type II von Neumann algebra is determined on traces
and applies more generally to positive elements in a C$^*$-algebra
with respect to the relation of Cuntz subequivalence.
The notion of tracial rank, which has its roots in work of Gong \cite{EllGon96} and Popa \cite{Pop97}
and was formalized and applied by Lin
in his seminal work of the 1990s as a way to circumvent inductive limit hypotheses 
in the stably finite case \cite{Lin01a,Lin01b,Lin04}, 
continues to do much of the technical legwork in classification.
The simple unital projectionless C$^*$-algebra $\cZ$ was 
introduced in the 1990s by Jiang and Su, who observed the parallel between its tensorial behaviour
and that of the hyperfinite II$_1$ factor $R$ \cite{JiaSu99}.

Winter's approach to classification, which was developed in the 2000s and has greatly impacted the 
course of the subject \cite{Win14}, made novel use of the operation of tensoring with $\cZ$, 
rendering greater urgency to the problem of recognizing when a C$^*$-algebra is $\cZ$-stable
and strengthening the analogy with $R$ through the latter's use in Connes's classification work,
which served as an inspiration.
Winter was also the first to realize the significance of
dimensional invariants based on nuclearity-type finite-dimensional approximation,
among which nuclear dimension has become the most eminent, and 
the connection between such invariants and $\cZ$-stability has become a centerpiece of his program.
In fact, it is a conjecture of Toms and Winter that for infinite-dimensional 
simple separable unital nuclear C$^*$-algebras the following three conditions are equivalent:
\begin{enumerate}
\item finite nuclear dimension, 

\item $\cZ$-stability,

\item strict comparison.
\end{enumerate}
The implications (i)$\Rightarrow$(ii) and (ii)$\Rightarrow$(iii) are theorems of Winter \cite{Win12}
and R{\o}rdam \cite{Ror04}, respectively. Matui and Sato proved (iii)$\Rightarrow$(ii) when
the set of extremal tracial states is finite and nonempty \cite{MatSat12},
and this was later generalized to the case where the extreme traces form a nonempty compact set 
with finite covering dimension \cite{KirRor14,Sat12,TomWhiWin15}.
The implication (ii)$\Rightarrow$(i) was first established by Sato, White, and Winter
in the case of a unique tracial state \cite{SatWhiWin15} and then more generally by
Bosa, Brown, Sato, Tikuisis, White, and Winter when the extreme tracial states form a nonempty compact set
\cite{BosBroSatTikWhiWin16}.
Thus the Toms--Winter conjecture has been fully confirmed in the case that the extreme tracial states form 
a nonempty compact set with finite covering dimension, and in particular when there is a unique tracial state.

The goal of these notes is to promote the development of a dynamical version of this theory 
surrounding the Toms--Winter conjecture, including connections to the C$^*$-algebra side
via the crossed product. This program requires first of all identifying the appropriate analogues of
nuclear dimension, strict comparison, and $\cZ$-stability.
There is a natural dynamical version of strict comparison 
which has appeared in lectures of Winter 
and has been studied by Buck in the case $G=\Zb$ \cite{Buc13}
(see also \cite{GlaWei95} for an earlier application of this concept to minimal
transformations of the Cantor set).
In parallel with \cite{Win10}, we simply refer to it as {\it comparison},
and also define the useful higher-order notions of $m$-comparison for integers $m\geq 0$,
with comparison representing the case $m=0$ (Definition~\ref{D-comparison}).
There are also by now several analogues of nuclear dimension, including the dynamic asymptotic
dimension and amenability dimension of Guentner, Willett, and Yu \cite{GueWilYu17}, 
and we will introduce here another, called
{\it tower dimension}, whose connections to nuclear dimension and dynamical comparison
are particularly stark, as shown in Sections~\ref{S-towdim nucdim} and \ref{S-towdim comparison}.
Although dynamic asymptotic dimension, amenability dimension, and tower dimension
do not coincide in general, there are inequalities relating them in the finite-dimensional case, 
and they are all equal when the space is zero-dimensional (see Section~\ref{S-dad}).

What has been missing is a dynamical substitute for $\cZ$-stability.
We introduce here a notion of {\it almost finiteness} for group actions on compact
metrizable spaces that will play the role of $\cZ$-stability in the Toms--Winter conjecture and 
of hyperfiniteness in the p.m.p.\ setting. We have adopted the terminology from Matui's
almost finiteness for groupoids, seeing that in the case of free actions on zero-dimensional
compact metrizable spaces our definition
reduces to Matui's (Section~\ref{S-zero dim}). As a comparison with the measure-theoretic framework, 
we recall that, for a free p.m.p.\ action $G\curvearrowright (X,\mu )$
of a countable amenable group, we can express the property of hyperfiniteness, in accordance
with the original proof of Ornstein and Weiss,
by saying that for every $\eps > 0$ there are measurable sets $V_1 , \dots , V_n \subseteq X$
and finite sets $S_1 , \dots , S_n \subseteq G$ with prescribed approximate invariance 
(in the F{\o}lner sense)
such that 
\begin{enumerate}
\item the sets $sV_i$ for $i=1,\dots ,n$ and $s\in S_i$ are pairwise disjoint, and

\item $\mu (X\setminus \bigsqcup_{i=1}^n S_i V_i ) < \eps$.
\end{enumerate}
The pair $(V_i ,S_i )$ we refer to as a {\it tower}, the set $V_i$ as the {\it base} of the tower,
the set $S_i$ as the {\it shape} of the tower, and the sets $sV_i$ for $s\in S_i$ as the {\it levels}
of the tower.
In the definition of almost finiteness, the sets $V_i$ are replaced by open sets and the smallness
of the remainder in (ii) is expressed topologically in terms of comparison with a portion of the
tower levels. Note in particular that almost finiteness implies that the acting group is amenable
because of the F{\o}lner requirement on the shapes of the towers.
In Theorem~\ref{T-af Z-stable} we prove that, for actions of countably infinite
groups on compact metrizable spaces, almost finiteness implies that the crossed product
is $\cZ$-stable. As we discuss at the end of Section~\ref{S-af Z-stable},
this can be used to give new examples of classifiable crossed products
for which dynamical techniques connected to nuclear dimension 
(such as in \cite{GueWilYu17} or Section~\ref{S-towdim nucdim}) are inapplicable 
due to finite-dimensionality
requirements on the space. What is particularly novel from the classification perspective
is that many of these examples can exhibit both infinite asymptotic dimension in the group
and positive topological entropy in the dynamics.

It is important to point out that almost finiteness is not an analogue of $\cZ$-stability by itself,
but rather of the conjunction of $\cZ$-stability and nuclearity. In view of 
classification theory, this combination (or its conjectural Toms--Winter equivalent, finite nuclear dimension)
could be argued to be the true topological analogue 
of hyperfiniteness, as opposed to just nuclearity, which is the direct technical translation 
of hyperfiniteness into the realm of C$^*$-algebras and as such is an essentially
measure-theoretic property. The interpretation of
almost finiteness as a combination of $\cZ$-stability and nuclearity is illustrated at a technical
level in the proof of Theorem~\ref{T-af Z-stable}, which relies on a
criterion for $\cZ$-stability that is special to the nuclear setting, 
due to Hirshberg and Orovitz (Theorem~\ref{T-Z-stable}).
The general characterization of $\cZ$-stability from which the Hirshberg--Orovitz result is derived
(Proposition~2.3 of \cite{Win10}) involves an additional approximate centrality requirement
that does not seem to translate into dynamical terms, and in particular does not seem to be amenable 
to the kind of tiling techniques that are integral to the proof of Theorem~\ref{T-af Z-stable}.

Consider now the following triad of properties for a free minimal action $G\curvearrowright X$
of a countably infinite amenable group on a compact metrizable space:
\begin{enumerate}
\item finite tower dimension,

\item almost finiteness,

\item comparison.
\end{enumerate}
In Theorem~\ref{T-comparison af} we establish the implication 
(ii)$\Rightarrow$(iii), as well as the converse (iii)$\Rightarrow$(ii)
in the case that the set $E_G (X)$ of ergodic $G$-invariant Borel probability measures is finite.
This precisely parallels the results of R{\o}rdam \cite{Ror04} and Matui--Sato \cite{MatSat12} 
mentioned above. Moreover, the argument for (iii)$\Rightarrow$(ii), like that of Matui and Sato, 
relies on an appeal to measure-theoretic structure, 
which in our case is the Ornstein--Weiss tower theorem.
In our proof of (iii)$\Rightarrow$(ii) it is enough that the action have
$m$-comparison for some $m\geq 0$, which is important as we also prove in Theorem~\ref{T-tower}
that if the covering dimension $\dim (X)$ is finite then (i) implies $m$-comparison 
for some $m\geq 0$,
and hence comparison in the case that $E_G (X)$ is finite.
Thus if $E_G (X)$ and $\dim (X)$ are both finite then we have 
(i)$\Rightarrow$(ii)$\Leftrightarrow$(iii), which we record as 
Theorem~\ref{T-comparison af tower}.

In \cite{Mat12} Matui showed that his property of almost finiteness for groupoids
has several nice consequences for the homology of the groupoid and its relation to 
both the topological full group and the $K$-theory of the reduced groupoid C$^*$-algebra.
In particular, if the groupoid is furthermore assumed to be principal (which amounts to
freeness in the case of actions) then
the first homology group is canonically isomorphic 
to the quotient of the topological full group by the subgroup generated by the elements
of finite order. Matui observes in Lemma~6.3 of \cite{Mat12} that 
the groupoids associated to free actions of $\Zb^d$ 
on zero-dimensional compact metrizable spaces are almost finite.
By combining the work of Szab{\'o}, Wu, and Zacharias in \cite{SzaWuZac17} with 
Theorems~\ref{T-zero dim} and Theorem~\ref{T-comparison af tower} we deduce
that this also holds for free minimal actions $G\curvearrowright X$
of finitely generated nilpotent groups on zero-dimensional compact metrizable spaces 
with $E_G (X)$ finite (Remark~\ref{R-nilpotent}).

While our results suggest that almost finiteness and comparison are full-fledged
dynamical analogues of their Toms--Winter counterparts, tower dimension and
its relatives unfortunately fall short on this account, despite their utility
in establishing finite nuclear dimension for crossed products of large classes of actions.
The problem is that tower dimension, dynamical asymptotic dimension, and amenability dimension
are too much affected by the dimensionality of the acting group and too little affected by the
dimensionality of the space and its interaction with the dynamics
(as captured by an invariant like mean dimension). On the side of the space, if we drop the assumption 
of finite-dimensionality then the implication (i)$\Rightarrow$(ii) fails, even for $G = \Zb$ 
(Example~\ref{E-not af}).
One can attempt to rectify this 
by imposing a small diameter condition on the tower levels in the definition
of tower dimension (we call the resulting invariant the {\it fine tower dimension}) but one would
not gain anything in the effort to relate dimensional invariants to almost 
finiteness and comparison since finite fine tower dimension already implies that $\dim (X)$ is finite.
Even more serious is the structural restriction imposed from the side of the group:
the tower dimension, dynamical asymptotic dimension, 
and amenability dimension are always infinite whenever $G$ has infinite asymptotic dimension,
which occurs frequently in the amenable case, an example being the Grigorchuk group.
In contrast, a generic free minimal action of any countably infinite amenable group
on the Cantor set is almost finite \cite{ConJacKerMarSewTuc17}.
Given that tower dimension seems as close as we can come in dynamics to being able to 
formally mimic the definition of nuclear dimension, and that it connects naturally to 
dynamical $m$-comparison and nuclear dimension in one direction of logical implication, 
we will perhaps have to be content with the prospect that the Toms--Winter conjecture
cannot be fully analogized within the coordinatized framework of group actions.
On the other hand, it is conceivable that the tower dimension of a free minimal action 
$G\curvearrowright X$ is always finite when $G$ is amenable and has finite asymptotic dimension 
and $X$ has finite covering dimension. This is indeed what happens if we furthermore
assume $G$ to be finitely generated and nilpotent (see Example~\ref{E-nilpotent}).

One more curious fact worth mentioning here is the possibility, suggested by 
the work of Elliott and Niu \cite{EllNiu14},
that for free minimal actions the small boundary property (or, alternatively, zero mean dimension)
is equivalent to $\cZ$-stability of the crossed product. 
Elliott and Niu showed that for free minimal $\Zb$-actions
the small boundary property (which is equivalent to zero mean dimension in this case)
implies $\cZ$-stability. The small boundary property and mean dimension are
formally very different from either nuclear dimension or $\cZ$-stability
and are more akin to slow dimension growth in inductive limit C$^*$-algebras,
as demonstrated by the proof in \cite{EllNiu14}, which employs arguments from an article 
of Toms on the equivalence of slow dimension growth and $\cZ$-stability for 
unital simple ASH algebras \cite{Tom11}.

We begin in Section~\ref{S-notation} by laying down some basic notation and terminology
used throughout the paper. In Section~\ref{S-comparison} we define (dynamical) comparison,
and also more generally $m$-comparison. Section~\ref{S-towdim} introduces tower dimension
and Section~\ref{S-dad} establishes inequalities relating it to dynamical asymptotic dimension 
and amenability dimension. In Section~\ref{S-towdim nucdim} we show how to derive an upper 
bound for the nuclear dimension of the crossed product of a free action of an amenable group
in terms of the tower dimension of the action and the covering dimension of the space.
In Section~\ref{S-towdim comparison} we prove that if the acting group is amenable and 
the tower dimension and covering dimension are both finite, with values $d$ and $c$, 
then the action has $((c+1)(d+1)-1)$-comparison.
In Section~\ref{S-af} we introduce almost finiteness and in Section~\ref{S-af comparison}
we establish Theorem~\ref{T-comparison af} relating it to comparison.
In Section~\ref{S-zero dim} we prove
that, for free actions on the Cantor set, almost finiteness is equivalent to having
clopen tower decompositions of the space with almost invariant shapes, so that it
reduces to Matui's notion of almost finiteness in this setting.
The behaviour of almost finiteness under extensions is investigated in Section~\ref{S-af extn}.
In Section~\ref{S-af Z-stable} we show that almost finiteness implies $\cZ$-stability
and use it to give new examples of classifiable crossed products.
Finally, in Section~\ref{S-type} we prove that, for free minimal actions 
of an amenable group on the Cantor set, almost finiteness implies that the clopen
type semigroup is almost unperforated, that this almost unperforation
in turn implies comparison, and that all three of these properties are equivalent
when the set $E_G (X)$ of ergodic $G$-invariant Borel probability measures is finite.
\medskip

\noindent{\it Acknowledgements.}
The author was partially supported by NSF grant DMS-1500593. He thanks
George Elliott, Xin Ma, and the referee for comments and corrections, and G{\'a}bor Szab{\'o} and Jianchao Wu 
for helpful discussions.

\section{General notation and terminology}\label{S-notation}

Throughout the paper $G$ is a countable discrete group.

For a compact Hausdorff space $X$, we write $C(X)$
for the unital C$^*$-algebra of continuous complex-valued functions on $X$.
For an open set $V\subseteq X$ we denote by $C_0 (V)$ the C$^*$-algebra of 
continuous complex-valued functions on $V$ which vanish at infinity,
which can be naturally viewed as a sub-C$^*$-algebra of $C(X)$.
We write $M(X)$ for the
convex set of all regular Borel probability measures on $X$, which is 
compact as a subset of the dual $C(X)^*$ equipped with the weak$^*$ topology. 
We denote the indicator function of a set $A\subseteq X$ by $\unit_A$.
The covering dimension of $X$ is written $\dim (X)$.

Actions on compact Hausdorff spaces are always assumed to be continuous.
Let $G\curvearrowright X$ be such an action. The image of a point $x\in X$ under a
group element $s$ is expressed as $sx$. For $A\subseteq X$, $s\in G$, and $K\subseteq G$
we write $sA = \{ sx : x\in A \}$ and $KA = \{ sx : s\in K,\, x\in A \}$. 
We write $M_G (X)$ for the convex set of $G$-invariant regular Borel probability measures on $X$,
which is a weak$^*$ compact subset of $M(X)$. We write $E_G (X)$ for the set of extreme points
of $M_G (X)$, which are precisely the ergodic measures in $M_G (X)$.

The {\it chromatic number} of a family $\cC$ of subsets of a given set is defined as
the least $d\in\Nb$ such that there is a partition of $\cC$ into $d$ subcollections
each of which is disjoint.

For any of the various notions of dimension which will appear, we will add a superscript
$+1$ to denote the value of the dimension plus one, so that $\dimplus (X) = \dim (X) + 1$, 
for example. 
This ``denormalization'' serves to streamline many formulas.

\section{Comparison and $m$-comparison}\label{S-comparison}

Throughout $G\curvearrowright X$ is an action on a compact metrizable space.

\begin{definition}
Let $m\in\Nb$. Let $A,B\subseteq X$. We write 
$A\prec_m B$ if for every closed set $C\subseteq A$ there exist a finite collection $\cU$
of open subsets of $X$ which cover $C$, an $s_U \in G$ for each $U\in\cU$, and a partition 
$\cU = \cU_0 \sqcup\cdots\sqcup \cU_m$ such that for each $i=0,\dots ,m$ the images 
$s_U U$ for $U\in\cU_i$ are pairwise disjoint subsets of $B$.
When $m=0$ we also write $A\prec B$.
\end{definition}

Note that the relation $\prec$ is transitive, as is straightforward to check.

\begin{definition}\label{D-comparison}
Let $m\in\Nb$. The action $G\curvearrowright X$ 
is said to have {\it $m$-comparison} if $A\prec_m B$ for all nonempty open sets $A,B\subseteq X$ 
satisfying $\mu (A) < \mu (B)$ for every $\mu\in M_G (X)$.
When $m=0$ we will also simply say that the action has {\it comparison}.
\end{definition}

The condition of nonemptiness on $A$ and $B$ above is included so as to cover the situation when
$M_G (X)$ is empty and can otherwise be dropped, as for example when $G$ is amenable.

The following lemma will be used repeatedly throughout the paper
and will be needed here to verify Proposition~\ref{P-open closed}.

\begin{lemma}\label{L-portmanteau}
Let $X$ be a compact metrizable space with compatible metric $d$
and let $\Omega$ be a weak$^*$ closed subset of $M(X)$.
Let $A$ be a closed subset of $X$ and $B$ an open subset of $X$
such that $\mu (A) < \mu (B)$ for all $\mu\in\Omega$. Then there exists
an $\eta > 0$ such that the sets 
\begin{align*}
B_- &= \{ x\in X : d(x,X\setminus B) > \eta \} , \\
A_+ &= \{ x\in X : d(x,A) \leq \eta \}
\end{align*}
satisfy $\mu (A_+ ) + \eta \leq \mu (B_- )$ for all $\mu\in \Omega$.
\end{lemma}

\begin{proof}
Suppose that the conclusion does not hold.
Then for every $n\in\Nb$ we can find a $\mu_n \in \Omega$
such that the sets 
\begin{align*}
B_n &= \{ x\in X : d(x,X\setminus B) > 1/n \} , \\
A_n &= \{ x\in X : d(x,A) \leq 1/n \}
\end{align*}
satisfy $\mu_n (A_n ) + 1/n > \mu_n (B_n )$. 
By the compactness of $\Omega$ there is a subsequence $\{ \mu_{n_k} \}$ of $\{ \mu_n \}$ 
which weak$^*$ converges to some $\mu\in\Omega$.
For a fixed $j\in\Nb$ we have, for every $k\geq j$,
\begin{align*}
\mu_{n_k} (A_{n_j} ) + \frac{1}{n_k} 
\geq \mu_{n_k} (A_{n_k} ) + \frac{1}{n_k}
> \mu_{n_k} (B_{n_k} ) 
\geq \mu_{n_k} (B_{n_j} ) ,
\end{align*}
and since $A_{n_j}$ is closed and $B_{n_j}$ is open the portmanteau theorem 
(\cite{Kec95}, Theorem~17.20) then yields
\begin{align*}
\mu (A_{n_j} ) 
\geq \limsup_{k\to\infty} \mu_{n_k} (A_{n_j}) 
\geq \liminf_{k\to\infty} \mu_{n_k} (B_{n_j} ) 
\geq \mu (B_{n_j} ) .
\end{align*}
Note that $B$ is equal to the increasing union of the sets $B_{n_j}$
for $j\in\Nb$, while $A$ is equal to the decreasing intersection
of the sets $A_{n_j}$ for $j\in\Nb$. Thus
\begin{align*}
\mu (A) 
= \lim_{j\to\infty} \mu (A_{n_j} ) 
\geq \lim_{j\to\infty} \mu (B_{n_j} )
= \mu (B) ,
\end{align*}
contradicting our hypothesis.
\end{proof}

In practice, we will use the following characterization as our effective definition
of $m$-comparison, usually without saying so.

\begin{proposition}\label{P-open closed}
Let $m\in\Nb$. The action $G\curvearrowright X$ has $m$-comparison if and only if $A\prec_m B$
for every closed set $A\subseteq X$ and nonempty open set $B\subseteq X$ 
satisfying $\mu (A) < \mu (B)$ for all $\mu\in M_G (X)$.
\end{proposition}

\begin{proof}
For the nontrivial direction, suppose that the action has $m$-comparison.
Let $A$ be a closed subset of $X$ and $B$ a nonempty open subset of $X$  
such that $\mu (A) < \mu (B)$ for all $\mu\in M_G (X)$.
Fixing a compatible metric $d$ on $X$,
by Lemma~\ref{L-portmanteau} there is an $\eta > 0$ such that the open
set $A' = \{ x\in X : d(x,A) < \eta \}$ satisfies $\mu (A') < \mu (B)$ for all $\mu\in M_G (X)$.
Then $A' \prec_m B$ by $m$-comparison, and so $A \prec_m B$, as desired.
\end{proof}

The remainder of the section is aimed at showing
that if $X$ is zero-dimensional then we can express comparison
using clopen sets and clopen partitions, as asserted by Proposition~\ref{P-clopen comparison}.

\begin{proposition}\label{P-clopen subequivalence}
Suppose that $X$ is zero-dimensional. Let $m\in\Nb$, and let $A$ and $B$ be clopen subsets of $X$.
Then $A\prec_m B$ if and only if there exist
a clopen partition $\cP$ of $A$, an $s_U \in G$ for every $U\in\cP$, and a partition
$\cP = \cP_0 \sqcup\cdots\sqcup \cP_m$ such that for each $i=0,\dots ,m$
the images $s_U U$ for $U\in\cP_i$ are pairwise disjoint subsets of $B$.
\end{proposition}

\begin{proof}
For the nontrivial direction, suppose that $A\prec_m B$. 
Then there exist $j_0 , \dots , j_m \in\Nb$,
open sets $U_{i,j}$ for $0\leq i\leq m$ and $1\leq j\leq j_i$ which 
cover $A$, and $s_{i,j} \in G$ such that for each $i=0,\dots ,m$
the images $s_{i,j} U_{i,j}$ for $j=1,\dots , j_i$ are pairwise disjoint subsets of $B$.
By the normality of $X$ we can then find, for all $i,j$,
a closed set $C_{i,j} \subseteq U_{i,j}$ such that the sets $C_{i,j}$ for all $i,j$ still cover $A$.
By compactness and zero-dimensionality, for given $i,j$ we can produce finitely many clopen sets 
contained in $U_{i,j}$ which cover $C_{i,j}$, and so we may assume that $C_{i,j}$ is clopen
by replacing it with the union of these clopen sets.
We now recursively define,
with respect to the lexicographic order on the pairs $i,j$,
\begin{align*}
A_{i,j} = \bigg(\bigg(A\setminus \bigsqcup_{k=0}^{i-1} \bigsqcup_{l=1}^{j_i} A_{k,l} \bigg) 
\cap C_{i,j} \bigg)\setminus (C_{i,1} \cup\cdots\cup C_{i,j-1} ) .
\end{align*}
These sets form a clopen partition of $A$ and for each $i=0,\dots ,m$ the images 
$s_{i,j} A_{i,j}$ for $j=1,\dots , j_i$ are pairwise disjoint subsets of $B$, as desired.
\end{proof}

\begin{proposition}\label{P-clopen comparison}
Suppose that $X$ is zero-dimensional. Let $m\in\Nb$.
Then the action $G\curvearrowright X$ has
$m$-comparison if and only if for all nonempty clopen sets $A,B\subseteq X$ 
satisfying $\mu (A) < \mu (B)$ for every $\mu\in M_G (X)$ there exist
a clopen partition $\cP$ of $A$, an $s_U \in G$ for every $U\in\cP$, and a partition
$\cP = \cP_0 \sqcup\cdots\sqcup \cP_m$ such that for each $0=1,\dots ,m$
the images $s_U U$ for $U\in\cP_i$ are pairwise disjoint subsets of $B$.
\end{proposition}

\begin{proof}
The forward direction is immediate from Proposition~\ref{P-clopen subequivalence}.
Suppose conversely that the action satisfies the condition in the proposition statement
involving clopen sets and let us establish $m$-comparison. Let $A$ be a closed subset of $X$
and $B$ a nonempty open subset of $X$ satisfying $\mu (A) < \mu (B)$ for all $\mu\in M_G (X)$.
By Lemma~\ref{L-portmanteau} there exists an $\eta > 0$ such that the sets
\begin{align*}
B_- &= \{ x\in X : d(x,X\setminus B) > \eta \} , \\
A_+ &= \{ x\in X : d(x,A) \leq \eta \}
\end{align*}
satisfy $\mu (A_+ ) < \mu (B_- )$ for all $\mu\in \Omega$. By an argument as
in the proof of Proposition~\ref{P-clopen subequivalence}, we can find clopen sets $A' , B' \subseteq X$
such that $A\subseteq A' \subseteq A_+$ and $B_- \subseteq B' \subseteq B$,
in which case $\mu (A' ) \leq \mu (A_+ ) < \mu (B_- ) \leq \mu (B' )$ for all $\mu\in \Omega$.
It follows by our hypothesis that there exist
a clopen partition $\cP$ of $A'$, an $s_U$ for every $U\in\cP$, and a partition
$\cP = \cP_0 \sqcup\cdots\sqcup \cP_m$ such that for each $i=0,\dots ,m$
the images $s_U U$ for $U\in\cP_i$ are pairwise disjoint subsets of $B'$.
Since $A\subseteq A'$ and $B' \subseteq B$,
we conclude (by Proposition~\ref{P-open closed}) that the action has $m$-comparison.
\end{proof}

\section{Tower dimension}\label{S-towdim}

Throughout $G\curvearrowright X$ is a free action on a compact Hausdorff space.

\begin{definition}
A {\it tower} is a pair $(V,S)$ consisting of a subset $V$ of $X$ and a finite subset $S$ of $G$
such that the sets $sV$ for $s\in S$ are pairwise disjoint. The set $V$ is the {\it base} of the tower,
the set $S$ is the {\it shape} of the tower, and the sets $sV$ for $s\in S$ 
are the {\it levels} of the tower.
We say that the tower $(V,S)$ is {\it open} if $V$ is open, {\it clopen} if $V$ is clopen,
and {\it measurable} if $V$ is measurable.
A collection of towers $\{ (V_i ,S_i ) \}_{i\in I}$ is said to {\it cover} $X$ 
if $\bigcup_{i\in I} S_i V_i = X$.
\end{definition}

\begin{definition}
Let $E$ be a finite subset of $G$. A collection of towers $\{ (V_i ,S_i ) \}_{i\in I}$ covering $X$
is {\it $E$-Lebesgue} if for every $x\in X$ there are an $i \in I$ and a $t\in S_i$ such that 
$x\in tV_i$ and $Et\subseteq S_i$.
\end{definition}

\begin{definition}\label{D-towdim}
The {\it tower dimension} $\towdim (X,G)$ of the action $G\curvearrowright X$ 
is the least integer $d\geq 0$ 
with the property that for every finite set $E\subseteq G$ 
there is an $E$-Lebesgue collection of open towers $\{ (V_i ,S_i) \}_{i\in I}$ covering $X$
such that the family $\{ S_i V_i \}_{i\in I}$ has chromatic number at most $d+1$.
If no such $d$ exists we set $\towdim (X,G) = \infty$.
\end{definition}

In the above definition one may assume, whenever convenient, that for each $i$ the identity element 
$e$ is contained in $S_i$ (i.e., the base $V_i$ is actually a level of the tower), 
for one can choose a $t \in S_i$ (assuming that $S_i$ is nonempty, as we may)
and replace $S_i$ by $S_i t^{-1}$ and $V_i$ by $tV_i$.

Note that if $G$ is not locally finite then the tower dimension must be at least $1$,
for if $E$ is a symmetric finite subset of $G$ and $\{ (V_i ,S_i) \}_{i\in I}$ is 
an $E$-Lebesgue collection of towers for which the sets $S_i V_i$ partition $X$ 
then for each $i$ with $V_i \neq\emptyset$ the set $S_i$ contains 
$\langle E\rangle S_i$ where $\langle E\rangle$ is the subgroup of $G$ generated by $E$. 

\begin{remark}
When $X$ is zero-dimensional we can equivalently restrict to clopen towers in Definition~\ref{D-towdim},
since we can use normality to slightly shrink the base of each of the open towers $(V_i ,T_i )$ 
to a closed set without destroying the fact that the collection of towers covers $X$, and then use
compactness and zero-dimensionality to slightly enlarge each of these closed bases to a clopen base
which is contained in the corresponding original base.
\end{remark}

\begin{example}\label{E-Z Cantor}
Let $\Zb\curvearrowright X$ be a minimal action on the Cantor set. This is given
by $(n,x)\mapsto T^n x$ for some transformation $T$ and is automatically free. 
We can decompose $X$ into clopen towers by the following standard procedure.
Take a nonempty clopen
set $V\subseteq X$, and consider the first return map which assigns to each $x\in V$
the smallest $n_x \in\Nb$ for which $T^{n_x} x\in V$, which is well defined by minimality.
This map is continuous by the clopenness of $V$ and so there is a clopen partition
$\{ V_1 , \dots , V_k \}$ of $V$ and integers $1 \leq n_1 < n_2 < \dots < n_k$ such that 
for each $i$ the set of all points in $V$ with return time $n_i$ is equal to $V_i$.
Setting $S_i = \{ 0,\dots , n_i - 1\}$, we thus have a collection of clopen towers $\{ (V_i , S_i )\}_{i=1}^k$
such that the sets $S_i V_i$ are pairwise disjoint, and since the union $\bigsqcup_{i=1}^k S_i V_i$
is closed and $T$-invariant it must be equal to $X$ by minimality. 
The only problem is that this collection
will not satisfy the Lebesgue condition in the definition of tower dimension.
To remedy this, we produce a second collection of towers by taking the image of the original 
one under some power of $T$, 
and make sure that the the numbers $n_i$ are sufficiently large.
More precisely, let $E$ be a finite subset of $\Zb$ and choose an $N > 2\max_{n\in E} |n|$.
Since the action is free, by shrinking $V$ we can force $n_1$ to be much larger than $N$,
which will imply that the collection of towers 
$\{ (V_i , S_i )\}_{i=1}^k \cup \{ (T^{-N} V_i , S_i + N )\}_{i=1}^k$ is $E$-Lebesgue, as is easily verified.
Thus the tower dimension of the action is at most $1$, and hence equal to $1$ 
by the observation following Definition~\ref{D-towdim}.
\end{example}

The following is verified by taking the inverse images 
under the extension map $Y\to X$ of all of towers at play in the definition 
of tower dimension. 

\begin{proposition}\label{P-towdim}
Let $G\curvearrowright Y$ be a free action on a compact Hausdorff space which is an extension
of $G\curvearrowright X$, meaning that there is $G$-equivariant continuous surjection $Y\to X$. 
Then
\[
\towdim (Y,G)\leq \towdim (X,G) .
\]
\end{proposition}

\begin{example}\label{E-mdim}
It was shown in \cite{GioKer10} that there are free minimal $\Zb$-actions on compact metrizable spaces
such that the crossed product $C(X)\rtimes\Zb$ fails to be $\cZ$-stable. By Proposition~\ref{P-towdim}
the examples given there have tower dimension at most $1$ since they factor onto an odometer, which 
has tower dimension $1$ by Example~\ref{E-Z Cantor} (note that $1$ is always a lower bound
for the tower dimension of free $\Zb$-actions by the observation following Definition~\ref{D-towdim}).
\end{example}

We recall that the asymptotic dimension $\asdim (G)$ of the group $G$ \cite{Gro93} can be expressed
as the least integer $d\geq 0$ such that for every finite set $E\subseteq G$ 
there exists a family $\{ U_i \}_{i\in I}$ of subsets
of $G$ of multiplicity at most $d+1$ with the following properties: 
\begin{enumerate}
\item there exists a finite set $F\subseteq G$  
such that for every $i\in I$ there is a $t\in G$ with $U_i \subseteq Ft$, and

\item for each $t\in G$ there is an $i\in I$ for which $Et\subseteq U_i$ (Lebesgue condition).
\end{enumerate}
If no such $d$ exists then $\asdim (G)$ is declared to be infinite.
It is readily seen that the asymptotic dimension is zero if and only if the group is locally finite.
The asymptotic dimension of $\Zb^m$ for $m\in\Nb$ is equal to $m$,
while the asymptotic dimension of the free group $F_m$ for $m\in\Nb$ is equal to $1$.
An example of a finitely generated amenable group with infinite asymptotic dimension 
is the Grigorchuk group \cite{Smi07}.
See \cite{BelDra08} for a general reference on the subject.

The following inequality is a refinement of the observation in the second paragraph following Definition~\ref{D-towdim},
which can be rephrased as saying that $\towdim (X,G)$ is nonzero whenever $\asdim (G)$ is nonzero.

\begin{proposition}\label{P-asymptotic dimension}
$\towdim (X,G) \geq \asdim (G)$.
\end{proposition}

\begin{proof}
We may assume that $\towdim (X,G)$ is finite. 
Let $E$ be a finite subset of $G$. Setting $d = \towdim (X,G)$, we can then find an
$E$-Lebesgue collection of open towers $\{ (V_i ,S_i) \}_{i\in I}$ covering $X$
such that the family $\{ S_i V_i \}_{i\in I}$ has chromatic number at most $d+1$.
Pick an $x\in X$. For every $i\in I$ set $L_i = \{ s\in G : sx\in V_i \}$.
Then the family $\bigcup_{i\in I} \{ S_i t : t\in L_i \}$ of subsets of $G$
is readily seen to satisfy the conditions in the above formulation of asymptotic dimension
with respect to the set $E$.
\end{proof}

\begin{example}\label{E-nilpotent}
Let $m\in\Nb$. It follows easily from Theorems~3.8 and 4.6 of \cite{Sza15}
that for every $d\in\Nb$ there is a constant $C > 0$ such that
for every free action $\Zb^m \curvearrowright X$ on a compact metrizable space
with $\dim (X) \leq d$ one has 
\begin{align}\label{E-C}
\towdimplus (X,\Zb^m ) \leq C\cdot\dimplus (X) ,
\end{align}
and that we can also relax the hypothesis $\dim (X) \leq d$
by merely requiring that the action have the topological small boundary property with 
respect to $d$ (\cite{Sza15}, Definition~3.2).
The arguments in Section~7 of \cite{SzaWuZac17} show more generally that 
that if $G$ is finitely generated and nilpotent then there exists such a $C > 0$ 
such that (\ref{E-C}) holds for every free action $G\curvearrowright X$ 
on a compact metrizable space with $\dim (X) \leq d$.
\end{example}

Finally, we define a variant of tower dimension which requires that the bases of the towers
have small diameter.

\begin{definition}
The {\it fine tower dimension} $\ftowdim (X,G)$ of the action $G\curvearrowright X$ 
is the least integer $d\geq 0$ 
with the property that for every finite set $E\subseteq G$ and $\delta > 0$
there is an $E$-Lebesgue collection of open towers $\{ (V_i ,S_i) \}_{i\in I}$ covering $X$
such that $\diam (sV_i ) < \delta$ for all $i\in I$ and $s\in S_i$ 
and the family $\{ S_i V_i \}_{i\in I}$ has chromatic number at most $d+1$.
If no such $d$ exists we set $\ftowdim (X,G) = \infty$.
\end{definition}

\begin{proposition}\label{P-coarse towdim}
One has
\begin{align*}
\towdimplus (X,G) \leq \ftowdimplus (X,G) \leq \towdimplus (X,G)\cdot\dimplus (X) .
\end{align*}
In particular, $\ftowdim (X,G) < \infty$
if and only if $\towdim (X,G) < \infty$ and $\dim (X) < \infty$.
\end{proposition}

\begin{proof}
The first inequality is trivial. For the second, we may suppose that 
$\towdimplus (X,G)$ and $\dimplus (X)$ are both finite. Denote these numbers by $d$ and $c$,
respectively. Let $E$ be a finite subset of $G$ and $\delta > 0$. 
Then there is an $E$-Lebesgue collection 
of towers $\{ (V_i ,S_i) \}_{i\in I}$ covering $X$ such that the family 
$\{ S_i V_i \}_{i\in I}$ has chromatic number at most $d+1$. By normality
we can find open sets $U_i \subseteq X$ with $\overline{U_i} \subseteq V_i$
such that the family $\{ S_i U_i \}_{i\in I}$ still covers $X$. Since $X$ has covering dimension $c$,
by compactness we can find for each $i$ a collection $\{ V_{i,1} ,\dots ,V_{i,k_i} \}$
of open subsets of $U_i$ which covers $\overline{V_i}$, 
satisfies $\diam (sV_i ) < \delta$ for all $s\in S_i$, and
has chromatic number at most $c+1$. Then $\{ (V_{i,j} ,S_i) : i\in I ,\, 1\leq k\leq j_i \}$
is an $E$-Lebesgue collection of towers such that each level of each tower
has diameter less than $\delta$ 
and the family $\{ S_i V_i \}_{i\in I}$ has chromatic number at most $(d+1)(c+1)$.
This establishes the second inequality.
\end{proof}

\section{Tower dimension, amenability dimension, and dynamic asymptotic dimension}\label{S-dad}

Throughout $G\curvearrowright X$ is a free action on a compact metrizable space.
Our aim here is to establish inequalities connecting its
tower dimension, amenability dimension, and dynamic asymptotic dimension 
(Theorem~\ref{T-dimensions}). We will see in particular that when the space is zero-dimensional, 
all of these dimensions are equal (Corollary~\ref{C-dimensions Cantor}).

The notions of amenability dimension and dynamic asymptotic dimension are due to
Guentner, Willett, and Yu \cite{GueWilYu17} and are recalled in Definitions~\ref{D-amdim} 
and \ref{D-dad} (these do not require freeness or metrizability). 
After defining amenability dimension we establish an inequality
relating it to tower dimension in Theorem~\ref{T-towdim lower amdim}. We then turn to 
dynamic asymptotic dimension and prove some lemmas which will help us link it
to tower dimension in Theorem~\ref{T-dimensions}.

Write $\Delta (G)$ for the set of probability measures on $G$, and $\Delta_d (G)$
for the set of probability measures on $G$ whose support has cardinality at most $d+1$.
We view both as subsets of $\ell^1 (G)$.

\begin{definition}\label{D-amdim}
The {\it amenability dimension} $\amdim (X,G)$ of the action $G\curvearrowright X$ 
is the least integer $d\geq 0$
with the property that for every finite set $F\subseteq G$ and $\eps > 0$ there is a continuous 
map $\varphi : X\to\Delta_d (G)$ such that
\[
\sup_{x\in X} \| \varphi (sx) - s\varphi (x) \|_1 < \eps
\]
for all $s\in F$.
\end{definition}

If $G$ is finite then every action $G\curvearrowright X$
has amenability dimension at most $|G|$, since we may construct a $G$-invariant map by sending everything in $X$
to the uniform probability measure on $G$. More generally,
if $G$ is amenable we can construct an approximately invariant continuous map $\varphi : X\to\Delta (G)$  
by sending everything in $X$ to the uniform probability measure on a sufficiently left invariant
finite subset of $G$. However, when $G$ is infinite the cardinality of the supports of such maps will necessarily 
tend to infinity as the approximate invariance becomes better and better, 
and so to derive bounds for the amenability dimension in this case 
one must search for maps which are approximately equivariant for reasons other than 
approximate invariance.
Indeed the support constraint in the definition of amenability dimension 
results in phenomena that are qualitatively very different
from the approximate invariance we see in an amenable group and instead involve 
the presence of collections of towers as in the definition of tower dimension.

\begin{theorem}\label{T-towdim lower amdim}
The action $G\curvearrowright X$ satisfies
\[
\amdim (X,G) \leq \towdim (X,G) .
\]
\end{theorem}

\begin{proof}
We denote the induced action of $G$ on $C(X)$ by $\alpha$,
that is, $\alpha_s (f)(x) = f(s^{-1} x)$ for all $s\in G$, $f\in C(X)$, and $x\in X$.

We may assume that $\towdim (X,G)$ is finite, and we denote this number by $d$. 
Fix a compatible metric $d$ on $X$.
 
Let $F$ be a finite subset of $G$ and let $\eps > 0$. In order to verify the condition
in the definition of amenability dimension we may assume that 
$F^{-1} = F$ by replacing $F$ with $F\cup F^{-1}$, and also that $e\in F$.
Choose an integer $n > 1$ such that $(d+1)(d+2)/n < \eps$. 
By the definition of tower dimension, there is an $F^n$-Lebesgue 
collection of towers $\{ (V_i,S_i ) \}_{i\in I}$
such that $\{ S_i V_i \}_{i\in I}$ is a cover of $X$ with chromatic number at most $d+1$.

By the $F^n$-Lebesgue condition and a simple compactness argument 
we can find a $\delta > 0$ such that for every $x\in X$ there are an $i\in I$ and a $t\in S_i$ such that 
$d(x, X\setminus tV_i ) > \delta$ and $F^n t\subseteq S_i$.
For every $i\in I$ and $t\in S_i$ define the function $\hat{g}_{i,t} \in C(X)$ by 
\[
\hat{g}_{i,t} (x) = \min \{ 1,\delta^{-1} d(x,X\setminus tV_i)\} .
\]
For every $i\in I$ set 
\[
g_i = \max_{t\in S_i} \alpha_{t^{-1}} (\hat{g}_{i,t} ) ,
\]
and note that for $t\in S_i$ the support of the function $\alpha_t (g_i )$ is contained in $tV_i$.

Let $i\in I$. Set $B_{i,n} = \bigcap_{t\in F^n} tS_i$ and 
$B_{i,0} = G \setminus \bigcap_{t\in F} tS_i$.
For $k=1,\dots ,n-1$ set 
\[
B_{i,k} = \bigg(\bigcap_{t\in F^k} tS_i \bigg)\setminus\bigcap_{s\in F^{k+1}} tS_i . 
\]
The sets $B_{i,k}$ for $k=0,\dots ,n$ form a partition of $G$, and for all $t\in F$ we have
\begin{enumerate}
\item $tB_{i,0} \subseteq B_{i,0} \cup B_{i,1}$,

\item $tB_{i,k} \subseteq B_{i,k-1} \cup B_{i,k} \cup B_{i,k+1}$ for every $k=1,\dots ,n-1$, 

\item $tB_{i,n} \subseteq B_{i,n-1} \cup B_{i,n}$.
\end{enumerate}
For each $t\in G$ take $k$ such that $t\in B_{i,k}$ and define the function
\[
\hat{h}_{i,t} = \frac{k}{n} \alpha_t (g_i ) 
\]
in $C(X)$, and note that $|\hat{h}_{i,t} (sx) - \hat{h}_{i,s^{-1} t} (x)| \leq 1/n$
for all $x\in X$ and $s\in F$.

Now set $H = \sum_{i\in I} \sum_{t\in G} \hat{h}_{i,t}$. By our choice of $\delta$,
for every $x\in X$ there is an $i\in I$ and a $t\in S_i$ such that 
$d(x, X\setminus tV_i ) > \delta$ and $F^n t\subseteq S_i$, in which case $t\in B_{i,n}$ and hence, 
in view of the definition of $g_i$,
\begin{align*}
\hat{h}_{i,t} (x) = \alpha_t (g_i )(x) \geq \hat{g}_{i,t} (x) = 1 .
\end{align*}
This shows that $H\geq 1$. Setting 
\[
h_{i,t} = H^{-1} \hat{h}_{i,t} 
\]
for every $i\in I$ and $t\in G$, we then define a continuous map $\varphi : X\to \Delta_d (G)$ by
\[
\varphi (x)(t) = \sum_{i\in I} h_{i,t} (x) 
\]
for $x\in X$ and $t\in G$.

Since for each $x\in X$ the set of all $i\in I$ such that $x\in S_i V_i$
has cardinality at most $d+1$, for $s\in F$ the difference between the values of 
$H$ at $x$ and $sx$ is at most $(d+1)/n$.
Since $H \geq 1$, it follows that the difference between the values
of $H^{-1}$ at $x$ and $sx$ is also at most $(d+1)/n$.
Consequently for every $x\in X$, $s\in F$, and $t\in G$ we have
\begin{align*}
|h_{i,t} (sx) - h_{i,s^{-1} t} (x)|
&\leq H (sx)^{-1} \big| \hat{h}_{i,t} (sx) - \hat{h}_{i,s^{-1} t} (x)\big| \\
&\hspace*{20mm} \ + \big| H (sx)^{-1} - H (x)^{-1}\big| \hat{h}_{i,s^{-1} t} (x) \\
&\leq \frac{d+2}{n} ,
\end{align*}
while $h_{i,t} (sx) = h_{i,s^{-1} t} (x) = 0$ whenever $x\notin S_i V_i$.
Using again the fact that for each $x\in X$ the set of all $i\in I$ such that $x\in S_i V_i$
has cardinality at most $d+1$, it follows that for every $x\in X$ and $s\in F$ we have
\begin{align*}
\| \varphi (sx) - s\varphi (x) \|_1
&= \sum_{t\in G} |\varphi (sx)(t) - \varphi (x)(s^{-1} t)| \\
&\leq \sum_{t\in G} \sum_{i\in I} |h_{i,t} (sx) - h_{i,s^{-1} t} (x)| \\
&\leq (d+1)\bigg( \frac{d+2}{n} \bigg) < \eps ,
\end{align*}
from which we conclude that $\amdim (X,G) \leq d$. 
\end{proof}

\begin{definition}\label{D-dad}
The {\it dynamic asymptotic dimension} $\dad (X,G)$ of the action $G\curvearrowright X$ 
is the least integer $d\geq 0$ with the property that for every finite set $E\subseteq G$ 
there are a finite set $F\subseteq G$ and an open cover $\cU$ of $X$ of cardinality $d+1$ 
such that, for all $x\in X$ and $s_1 , \dots , s_n \in E$, if the points
$x, s_1 x, s_2 s_1 x , \dots , s_n \cdots s_1 x$ are contained in a common member of $\cU$
then $s_n \cdots s_1 \in F$.
\end{definition}

\begin{definition}
Let $E$ be a finite subset of $G$. An open cover $\cU$ of $X$ is said to be {\it $E$-Lebesgue} 
if for every $x\in X$ there is an $1\leq i\leq n$ such that $Ex \subseteq U_i$.
\end{definition}

\begin{remark}
The above definition should not be confused with the $E$-Lebesgue condition for a collection of towers.
Given a collection of towers $\cT = \{ (V_i ,S_i \}_{i\in I}$ 
such that the family $\cV = \{ S_i V_i \}_{i\in I}$
covers $X$, if $\cT$ is $E$-Lebesgue then $\cV$ is $E$-Lebesgue, but not conversely.
For example, if $E$ contains an element of infinite order then there is no 
tower $(V,S)$ such that the singleton $\{ (V,S) \}$ is $E$-Lebesgue and $SV = X$,
although $\{ X \}$ is an $E$-Lebesgue cover of $X$. For collections of towers 
the $E$-Lebesgue condition involves the way in which each tower is coordinatized by its shape,
while no such coordinatization is at play when dealing with members of an arbitrary cover.
\end{remark}

The following is part of Corollary~4.2 in \cite{GueWilYu17}.

\begin{proposition}\label{P-dad Lebesgue}
In Definition~\ref{D-dad} the open cover $\cU$ can be chosen to be $E$-Lebesgue. 
\end{proposition}

We next record some lemmas that will allow us to establish the inequality
$\ftowdimplus (X,G) \leq \dadplus (X,G)\cdot\dimplus (X)$ in Theorem~\ref{T-dimensions}.

\begin{definition}\label{D-castle 1}
Let $G\curvearrowright X$ be a free action on a compact metric space. 
A {\it castle} is a finite collection of towers $\{ (V_i , S_i) \}_{i\in I}$
such that the sets $S_i V_i$ for $i\in I$ are pairwise disjoint.
The {\it levels} of the castle are the sets $sV_i$ for $i\in I$ and $s\in S_i$.
We say that the castle is {\it open} if each of the towers is open.
\end{definition}

\begin{definition}\label{D-R}
For sets $W\subseteq X$ and $E\subseteq G$ we write $R_{W,E}$ for the equivalence relation
on $W$ under which two points $x$ and $y$ are equivalent if there exist 
$s_1 , \dots , s_n \in E\cup E^{-1} \cup \{ e \}$ such that $y = s_n \cdots s_1 x$ and 
$s_k \cdots s_1 x \in W$ for $k=1,\dots ,n-1$.
Note that $R_{W,E}$ is symmetric because the set $E\cup E^{-1} \cup \{ e \}$ is symmetric.
\end{definition}

For an equivalence relation $R$ on a set $Z$ and an $A\subseteq Z$
we write $[A]_R$ for the saturation of $A$, i.e., the set of all $x\in Z$
for which there exists a $y\in A$ such that $xRy$.
For sets $W,A\subseteq X$ we write $\partial_A W$ for the boundary of $A\cap W$ 
as a subset of the set $A$ equipped with the relative topology.

We will use without comment the following properties of covering dimension 
for a metrizable space $Y$. The second and third are consequences of the fact that 
covering dimension and large inductive dimension coincide in the metrizable setting.
See \cite{Nag83} for more information.
\begin{enumerate}
\item If $A$ is a closed subset of $Y$ then $\dim (A)\leq \dim (Y)$.

\item For every open set $U\subseteq Y$ 
and closed set $C\subseteq U$ there exists an open set $V\subseteq Y$
with $C\subseteq V\subseteq U$ and $\dim (\partial V) < \dim (Y)$.

\item If $\{ C_1 , \dots , C_n \}$ is a closed covering of $Y$ then
$\dim (Y)\leq \max_{i=1,\dots ,n} \dim (C_i )$.
\end{enumerate}

\begin{lemma}\label{L-small diam}
Let $A$ be a nonempty closed subset of $X$ and let $\delta > 0$. 
Then there is a finite collection $\{ B_1 , \dots , B_n  \}$
of pairwise disjoint relatively open subsets of $A$ of diameter less than $\delta$ 
such that the set $\bigsqcup_{j=1}^n B_j$ is dense in $A$ and $\dim (\partial_A B_j ) < \dim (A)$
for every $j=1,\dots ,n$.
\end{lemma}

\begin{proof}
By compactness there exists a finite open cover $\{ U_1 , \dots , U_n \}$ of $X$ 
whose members each have diameter less than $\delta$, and by normality we can find closed sets $C_j \subseteq U_j$
such that the collection $\{ C_1 , \dots , C_n \}$ is also a cover of $X$. Relativizing to $A$,
we can then find for each $j=1,\dots ,n$ a relatively open subset $V_j$ of $A$ such that
$C_j \cap A \subseteq V_j \subseteq U_j \cap A$ and $\dim (\partial_A V_j ) < \dim (A)$.
Now recursively define $B_1 = V_1$ and 
$B_j = V_j \setminus (\overline{V_1} \cup\cdots\cup\overline{V_{j-1}} )$ for $j=2,\dots , n$.
Then the set $B = \bigsqcup_{j=1}^n B_j$ is dense in $A$ and for every $j=1,\dots , n$ we have
\begin{align*}
\dim (\partial_A B_j ) 
&\leq \max \{\dim (\partial_A V_j ),\dim (\partial_A \overline {V_{j-1}} ) ,\dots , \dim (\partial_A \overline {V_1} )\} \\
&\leq \max \{\dim (\partial_A V_j ),\dim (\partial_A V_{j-1} ) ,\dots , \dim (\partial_A V_1 )\} \\
&< \dim (A) . \qedhere
\end{align*}
\end{proof}

\begin{lemma}\label{L-closed}
Let $E$ be a finite subset of $G$ with $E^{-1} = E$ and $e\in E$. 
Let $C$ be a closed subset of $X$, and suppose that there is a finite set $F\subseteq G$
such that, for all $x\in C$ and $s_1 , \dots , s_m \in E$, if 
$s_k \cdots s_1 x \in C$ for all $k=1,\dots ,m$ 
then $s_m \cdots s_1 \in F$. Let $A$ be a closed subset of $C$.
Then $[A]_{R_{C,E}}$ is closed.
\end{lemma}

\begin{proof}
Let $x$ be a point in $X$ which is the limit of  
some sequence $\{ x_n \}$ in $[A]_{R_{C,E}}$,
and let us show that $x\in [A]_{R_{C,E}}$.
Since $F$ is finite
we can assume, by passing to a subsequence, 
that there are $s_1 , \dots , s_m \in E$ and $a_n \in A$ such that
for every $n$ we have $x_n = s_m \cdots s_1 a_n$ and $s_k \cdots s_1 a_n \in C$
for $k=1,\dots ,m$. By the continuity of the action, we have
$a_n = s_1^{-1} \cdots s_m^{-1} x_n \to s_1^{-1} \cdots s_m^{-1} x$ as $n\to\infty$. 
Writing $a = s_1^{-1} \cdots s_m^{-1} x$, which belongs to $A$ since $A$ is closed,
we then have $x = s_m \cdots s_1 a$, and also
$s_k \cdots s_1 a \in C$ for $k=1,\dots ,m$ since $C$ is closed.
Thus $x$ and $a$ are $R_{C,E}$-equivalent, so that $x\in [A]_{R_{C,E}}$. 
We conclude that $[A]_{R_{C,E}}$ is closed.
\end{proof}

\begin{lemma}\label{L-lower dim}
Let $E$ be a finite subset of $G$ with $E^{-1} = E$ and $e\in E$. Let $\delta > 0$.
Let $U$ be an open subset of $X$ and $F$ a finite subset of $G$
such that, for all $x\in U$ and $s_1 , \dots , s_m \in E$, if 
$s_k \cdots s_1 x \in U$ for all $k=1,\dots ,m$ 
then $s_m \cdots s_1 \in F$. 
Let $C$ be a nonempty closed subset of $X$ such that $C\subseteq U$.
Let $A$ be a closed subset of $\overline{U}$ with $A = [A]_{R_{\overline{U},E}}$.
Then there are an open set $W\subseteq X$ with $C\subseteq W\subseteq \overline{W}\subseteq U$,
an open castle $\{ (V_i ,S_i ) \}_{i\in I}$, and sets $O_i \subseteq V_i$ such that 
\begin{enumerate}
\item $\diam (sV_i ) < \delta$ for all $i\in I$ and $s\in S_i$,

\item $\bigsqcup_{i\in I} S_i V_i \subseteq W$,

\item $[tx]_{R_{W,E}} = S_i x$ for every $i\in I$, $t\in S_i$, and $x\in O_i$, 

\item the set $(A\cap\overline{W} )\setminus \bigsqcup_{i\in I} S_i O_i$ is closed and
has dimension strictly less than $\dim (A)$.
\end{enumerate}
\end{lemma}

\begin{proof}
Take an open set $W_0 \subseteq X$ with $C\subseteq W_0 \subseteq \overline{W_0}\subseteq U$.
Then we can find a relatively open subset $W_1$ of $A$ with 
$A\cap C \subseteq W_1 \subseteq A\cap W_0$ such that $\dim (\partial_A W_1 ) < \dim (A)$. 
Now take an open set $W\subseteq W_0$ such that $W_1 = A\cap W$ and $C\subseteq W$, and note that
$\overline{W}\subseteq \overline{W_0} \subseteq U$.
Set
\[
X_0 = A \setminus [\partial_A W_1 ]_{R_{\overline{W},E}} 
\]
Since $A$ is closed the set $\partial_A W_1$ is closed, and thus
by Lemma~\ref{L-closed} the set $[\partial_A W_1 ]_{R_{\overline{W},E}}$ is closed, 
so that $X_0$ is relatively open in $A$. Moreover, since $[\partial_A W_1 ]_{R_{\overline{W},E}}$ is
contained in $F\partial_A W_1$ we have
\begin{align}\label{E-boundary dim}
\dim ([\partial_A W_1 ]_{R_{\overline{W},E}} )
\leq \dim (F\partial_A W_1 )
= \dim (\partial_A W_1 ) 
< \dim (A) .
\end{align}

Define a map $\varphi : X_0 \to \sP (F)$ (the power set of $F$) by 
\[
\varphi (x) = \{ s\in F : sx \in [x]_{R_{W,E}} \} ,
\]
which by freeness is determined by the equation
$\varphi (x) x = [x]_{R_{W,E}}$.
Let us verify that $\varphi$ is continuous. Let $x\in X_0$. 
Since $W$ is open and the action is continuous, we can find 
a relatively open subset
$V$ of $X_0$ containing $x$ such that $\varphi (x) \subseteq \varphi (y)$ for every 
$y\in V$. Suppose that there exists a sequence $\{ x_n \}$ in $X_0$ converging to $x$
such that $\varphi (x) \neq \varphi (x_n)$ for every $n$. 
We may assume, by passing to a subsequence, that there is a $t\in F$ such that
$t\notin \varphi (x)$ and $t\in \varphi (x_n)$ for every $n$.
Since the cardinality of each equivalence class of $R$ is bounded above by $|F|$,
we can also assume, by passing to a further subsequence,
that there are $s_1 , \dots , s_m \in E$ such that
$s_m \cdots s_1 = t$ and $s_k \cdots s_1 x_n \in W$ for $k=1,\dots ,m$.
Then by the continuity of the action 
we have $s_k \cdots s_1 x \in \overline{W}$ for $k=1,\dots ,m$.
Now if it were the case that $s_k \cdots s_1 x \notin W$ for some $1\leq k\leq m$,
then since  $s_k \cdots s_1 x \in [A]_{R_{\overline{U},E}} = A$
and $s_k \cdots s_1 x_n \in [A]_{R_{\overline{U},E}} \cap W = A\cap W$ for every $n$ 
it would follow that $s_k \cdots s_1 x \in\partial_A W_1$ and hence 
$x\in [\partial_A W_1 ]_{R_{\overline{W},E}}$,
contradicting the membership of $x$ in $X_0$. 
Therefore $s_k \cdots s_1 x \in W$ for $k=1,\dots ,m$, showing that 
$t\in\varphi (x)$, a contradiction. We conclude from this that $\varphi$
is constant on some open neighbourhood of $x$, and hence that $\varphi$ is continuous on $X_0$, 
as desired.

Enumerate the subsets of $F$ containing $e$ as $S_1 , \dots , S_q$.
Recursively define subsets $O_1 , \dots , O_q$ of $X_0$ by setting
$O_1 = \varphi^{-1} (S_1 )$ and, for $i=2,\dots ,q$,
\begin{align*}
O_i = \varphi^{-1} (S_i ) \setminus 
\big(\overline{S_{i-1} \varphi^{-1} (S_{i-1} )} \cup\cdots\cup \overline{S_1 \varphi^{-1} (S_1 )} \big) .
\end{align*}
The sets $S_i O_i$ for $i=1,\dots ,q$ are pairwise disjoint because $R_{W,E}$ is an equivalence relation.
Note also that each set $S_i O_i$ is contained in $X_0$ since $[A]_{R_{W,E}} \subseteq [A]_{R_{\overline{U},E}} = A$
and $[x]_{R_{W,E}} \subseteq [\partial_A W_1 ]_{R_{\overline{W},E}}$ for every $x\in [\partial_A W_1 ]_{R_{\overline{W},E}}$. Moreover, for every $i=1,\dots , q$ we have, using 
the relative openness of $\varphi^{-1} (S_i )$ in $A$ and the closedness of $A$ and
appealing to (\ref{E-boundary dim}),
\begin{align*}
\dim (\partial_A S_i \varphi^{-1} (S_i )) 
&= \dim (S_i \partial_A \varphi^{-1} (S_i )) \\
&= \dim (\partial_A \varphi^{-1} (S_i )) 
\leq \dim ([\partial_A W_1 ]_{R_{\overline{W},E}} ) 
< \dim (A)
\end{align*}
and hence
\begin{align*}
\dim (\partial_A O_i )
&\leq \max \big(\dim (\partial_A \varphi^{-1} (S_i )) , \dim (\partial_A \overline{S_{i-1} \varphi^{-1} (S_{i-1} )}) , \dots,
\dim (\partial_A \overline{S_1 \varphi^{-1} (S_1 )}) \big) \\
&\leq \max \big(\dim (\partial_A \varphi^{-1} (S_i )) , \dim (\partial_A S_{i-1} \varphi^{-1} (S_{i-1} )) , \dots,
\dim (\partial_A S_1 \varphi^{-1} (S_1 ))\big) \\
&< \dim (A) .
\end{align*}
We thus have a castle $\{ (O_i ,S_i ) \}_{1\leq i\leq q}$ with the following properties:
\begin{enumerate}
\item $\bigsqcup_{i=1}^q S_i O_i \subseteq X_0 \subseteq\bigsqcup_{i=1}^q \overline{S_i O_i}$,

\item $[tx]_{R_{W,E}} = S_i x$ for every $i=1,\dots ,q$, $t\in S_i$, and $x\in O_i$,

\item $\dim (\partial_A O_i ) < \dim (A)$ for every $i=1,\dots , q$.
\end{enumerate}
By Lemma~\ref{L-small diam} and uniform continuity there is a family $\{ B_1 , \dots , B_n  \}$
of pairwise disjoint relatively open subsets of $A$ 
such that the diameter of $sB_j$ is less than $\delta$ for every $j=1,\dots ,n$ and $s\in F$,
the set $\bigsqcup_{j=1}^n B_j$ is dense in $A$, and $\dim (\partial_A B_j ) < \dim (A)$
for every $j=1,\dots ,n$. 
Replacing the castle $\{ (O_i ,S_i ) \}_{1\leq i\leq q}$ (the details of whose construction we don't care about)
with the castle $\{ (O_i \cap B_j ,S_i ) \}_{1\leq i\leq q,\, 1\leq j \leq n}$ and relabeling,
we may assume that, in addition to satisfying (i) to (iii), the castle $\{ (O_i ,S_i ) \}_{1\leq i\leq q}$ 
has the property that all of its levels have diameter less than $\delta$
(to see that condition (iii) still holds observe that
$\dim (\partial_A (O_i \cap B_j )) \leq \dim (\partial_A O_i \cup \partial_A B_j ) 
\leq \max \{ \dim (\partial_A O_i ) , \dim (\partial_A B_j ) \} < \dim (A)$).

Next let $1\leq i\leq q$ and $s\in S_i$ and let us show that $sO_i$ is relatively open in $X_0$.
Suppose to the contrary that there exists a sequence $\{ y_n \}$ in $X_0 \setminus sO_i$
which converges to some $y\in sO_i$. Set $x=s^{-1} y \in O_i$. 
Then there are $s_1 , \dots , s_m \in E$ such that $s = s_m \dots s_1$ and $s_k \cdots s_1 x \in W$ for every $k=1,\dots ,m$.
For every $n$ set $x_n = s^{-1} y_n = s_1^{-1} \cdots s_m^{-1} y_n$, and note that $x_n \to x$ by the continuity of the action.
Since $W$ is open we may assume, by passing to subsequences, that for each $n$ we have $s_k \cdots s_1 x_n \in W$ 
for every $k=1,\dots ,m$, which means that $x_n$ and $y_n$ are $R_{W,E}$-equivalent. Since 
$y_n$ belongs to $A$ but not $[\partial_A W]_{R_{W,E}}$, 
this implies that $x_n \in [A]_{R_{W,E}} \subseteq [A]_{R_{\overline{U},E}} = A$
and $x_n \notin [\partial_A W]_{R_{W,E}}$. Therefore $x_n \in X_0$ for every $n$.
Since $x_n = s^{-1} y_n \notin O_i$ for every $n$ and $x_n \to x$, 
and $O_i$ is relatively open in $X_0$ by the continuity of $\varphi$, 
we thus arrive at a contradiction. We therefore conclude that $sO_i$ is relatively open in $X_0$.
It follows that for each $i=1,\dots ,q$ and $s\in S_i$ we can find an open subset
$V_{i,s}$ of $W$ which contains $sO_i$ and has diameter less than $\delta$ so that the sets $V_{i,s}$
are pairwise disjoint. For each $i=1,\dots ,q$ set $V_i = \bigcap_{s\in S_i} s^{-1} V_{i,s}$
Then $\{ (V_i ,S_i ) \}_{1\leq i\leq q}$ is an open castle such that $O_i \subseteq V_i$ for every $i=1,\dots ,q$
and
\begin{enumerate}
\item[(iv)] $\dim (sV_i ) < \delta$ for all $i=1,\dots ,q$ and $s\in S_i$,

\item[(v)] $\bigsqcup_{i=1}^q S_i V_i \subseteq W$.
\end{enumerate}

Define the set
\begin{align*}
D = S_1\partial_A O_1 \cup\dots\cup S_q \partial_A O_q ,
\end{align*}
which can be written as $\partial_A (S_1 O_1 )\cup\dots\cup \partial_A (S_q O_q )$ by the relative
openness of each $O_i$ in $A$ and thus by (i) satisfies $X_0 \setminus D \subseteq \bigsqcup_{i=1}^q S_i O_i$. 
Using (iii) we have
\begin{align}\label{E-dim max}
\dim (D) 
&\leq\max \{ \dim (S_1\partial_A O_1 ), \dots , \dim (S_q\partial_A O_q ) \} \\
&\leq\max \{ \dim (\partial_A O_1 ), \dots , \dim (\partial_A O_q ) \} \notag \\
&< \dim (A) . \notag
\end{align}

Finally, set $A' = (A\cap\overline{W} )\setminus \bigsqcup_{i=1}^q S_i O_i$, which is relatively closed in $A$ 
and hence closed in $X$.
Since $X_0 \setminus D \subseteq \bigsqcup_{i=1}^q S_i O_i$ and $A\cap \overline{W} = \overline{A\cap W} = \overline{W_1}$, 
we have
\begin{align*}
A' \subseteq (A\cap\overline{W}) \setminus (X_0 \setminus D) 
&\subseteq \partial_A W_1 \cup (W_1 \setminus X_0 ) \cup D \\
&\subseteq [\partial_A W_1 ]_{R_{\overline{W},E}} \cup D 
\end{align*}
and hence, using (\ref{E-boundary dim}) and (\ref{E-dim max}) and the fact that $[\partial_A W_1 ]_{R_{\overline{W},E}}$
and $D$ are relatively closed in $A$,
\begin{align*} 
\dim (A' ) 
\leq \max \{ \dim ([\partial_A W_1 ]_{R_{\overline{W},E}} ), \dim (D) \} 
< \dim (A) .
\end{align*}
This completes the verification of the required properties.
\end{proof}

\begin{lemma}\label{L-lower dim 2}
Let $E$, $\delta$, $U$, $F$, and $C$ be as in the statement of Lemma~\ref{L-lower dim}.
Then there are a nonnegative integer $d\leq \dim (X)$ and
for each $j=1,\dots , d+1$ an open castle $\{ (V_i , S_i )\}_{i\in I_j}$ and
sets $O_i \subseteq V_i$
such that 
\begin{enumerate}
\item $C\subseteq\bigcup_{j=1}^{d+1} \bigsqcup_{i\in I_j} S_i O_i 
\subseteq \bigcup_{j=1}^{d+1} \bigsqcup_{i\in I_j} S_i V_i \subseteq U$, 

\item $[x]_{R_{C,E}} \subseteq S_i x$ for every $i\in I$, $t\in S_i$, and $x\in tO_i \cap C$, and

\item for every $j=1,\dots , d+1$  one has $\diam (sV_i ) < \delta$ for all $i\in I_j$ and $s\in S_i$.
\end{enumerate}
\end{lemma}

\begin{proof}
Take an open set $U_0$ with $C\subseteq U_0 \subseteq \overline{U_0} \subseteq U$
and set $A_0 = [\overline{U_0} ]_{R_{\overline{U_0} ,E}}$. Then $A_0$ is closed by Lemma~\ref{L-closed},
and $A_0 = [A_0 ]_{R_{\overline{U_0} ,E}}$. Thus
by Lemma~\ref{L-lower dim} there is an open set $U_1$ 
with $C\subseteq U_1 \subseteq \overline{U_1} \subseteq U_0$,
an open castle $\{ (V_i ,S_i) \}_{i\in I_1}$, and sets $O_i \subseteq V_i$ such that
\begin{enumerate}
\item $\diam (sV_i ) < \delta$ for all $i\in I_1$ and $s\in S_i$,

\item $\bigsqcup_{i\in I_1} S_i V_i \subseteq U_1$,

\item $[tx]_{R_{U_1,E}} = S_i x$ for every $i\in I_1$, $t\in S_i$, and $x\in O_i$, and

\item the set $A_0' = (A_0 \cap \overline{U_0} )\setminus \bigsqcup_{i\in I_1} S_i O_i$ is closed and satisfies
\[
\dim (A_0' ) < \dim (A_0) .
\]
\end{enumerate}
Note that (iii) implies that for every $i\in I_1$, $t\in S_i$, and $x\in tO_i \cap C$
we have
\begin{align*}
[x]_{R_{C,E}} \subseteq [x]_{R_{U_1,E}} = S_i x .
\end{align*}

Set $A_1 = [A_0' ]_{R_{\overline{U_0} ,E}}$, which is closed by Lemma~\ref{L-closed}, and observe that
since $A_1 \subseteq FA_0'$ we have
\begin{align*}
\dim (A_1 ) \leq \dim (FA_0' ) = \dim (A_0' ) < \dim (A_0 ).
\end{align*}

Apply Lemma~\ref{L-lower dim} again, this time using $A_1$ and $U_1$, to get an open set $U_2$ 
with $C\subseteq U_2 \subseteq \overline{U_2} \subseteq U_1$, an open castle $\{ (V_i ,S_i) \}_{i\in I_2}$,
and sets $O_i \subseteq V_i$ such that
\begin{enumerate}
\item $\diam (sV_i ) < \delta$ for all $i\in I_2$ and $s\in S_i$,

\item $\bigsqcup_{i\in I_2} S_i V_i \subseteq U_2$,

\item $[tx]_{R_{U_2,E}} = S_i x$ for every $i\in I_1$, $t\in S_i$, and $x\in O_i$, and

\item the set $A_1' = (A_1 \cap \overline{U_1} )\setminus \bigsqcup_{i\in I_2} S_i O_i$ is closed and satisfies
\[
\dim (A_1' ) < \dim (A_1) .
\]
\end{enumerate}
Note that (iii) implies that for every $i\in I_2$, $t\in S_i$, and $x\in tO_i \cap C$
we have
\begin{align*}
[x]_{R_{C,E}} \subseteq [x]_{R_{U_2,E}} = S_i x .
\end{align*}

Set $A_2 = [A_1' ]_{R_{A_1,E}}$, which is closed by Lemma~\ref{L-closed}, and observe that
since $A_2 \subseteq FA_1'$ we have
\begin{align*}
\dim (A_2 ) \leq \dim (FA_1' ) = \dim (A_1' ) < \dim (A_1 ).
\end{align*}

Continue this procedure by recursively applying Lemma~\ref{L-lower dim} to produce at the $j$th stage
sets $U_j$ and $A_j$, an open castle $\{ (V_i ,S_i) \}_{i\in I_j}$, and sets $O_i \subseteq V_i$ as above, 
until we reach the point that 
$\dim (A_{d+1} ) = -1$
for some $d\leq \dim (X)$.
The castles $\{ (V_i ,S_i) \}_{i\in I_j}$ and sets $O_i$ then satisfy the 
requirements of the lemma.
\end{proof}

\begin{lemma}\label{L-Lebesgue}
Let $E$ be a finite subset of $G$ and let $\{ U_1 , \dots , U_n \}$ be an $E$-Lebesgue open cover of $X$.
Then there exist closed sets $C_i \subseteq U_i$ such that 
$\{ C_1 , \dots , C_n \}$ is an $E$-Lebesgue cover of $X$.
\end{lemma}

\begin{proof}
For $i=1,\dots , n$ write $V_i$ for the set of all $x\in X$ such that
$Ex\subseteq U_i$. This is an open set since $U_i$ is open and the action is continuous.
Because $\{ U_1 , \dots , U_n \}$ is $E$-Lebesgue, the collection $\{ V_1 , \dots , V_n \}$
covers $X$. By normality we can find closed sets $D_i \subseteq V_i$
such that $\{ D_1 , \dots , D_n \}$ still covers $X$. For each $i$
define $C_i = \bigcup_{s\in E} s^{-1} D_i$, which is a closed subset of $U_i$. Then 
$\{ C_1 , \dots , C_n \}$ is an $E$-Lebesgue cover of $X$ of the required kind.
\end{proof}

\begin{theorem}\label{T-dimensions}
The action $G\curvearrowright X$ satisfies
\begin{align*}
\dadplus (X,G) \leq \amdimplus (X,G) 
&\leq \towdimplus (X,G) \\
&\leq \ftowdimplus (X,G) \leq \dadplus (X,G)\cdot\dimplus (X) .
\end{align*}
\end{theorem}

\begin{proof}
The first inequality follows from Theorem~4.11 of \cite{GueWilYu17},
as pointed out in Remark~4.14 of that paper. The second inequality is Theorem~\ref{T-towdim lower amdim}.
The third inequality is trivial. 

It remains to establish the last inequality. For this we may assume that 
$\dad (X,G)$ and $\dim (X)$ are both finite.
Let $E$ be a finite subset of $G$ with $E^{-1} = E$ and $e\in E$.
By Proposition~\ref{P-dad Lebesgue} there are a finite set $F\subseteq G$ and 
an $E$-Lebesgue open cover $\{ U_0 , \dots , U_d \}$ of $X$ with $d\leq \dad (X,G)$
such that, for all $j=0,\dots, d$, if $x\in U_j$ and $s_1 , \dots , s_n \in E$ satisfy
$s_k \cdots s_1 x \in U_j$ for all $k=1,\dots , n$ 
then $s_n \cdots s_1 \in F$. By Lemma~\ref{L-Lebesgue} there exist closed sets
$C_j \subseteq U_j$ for $j=0,\dots , d$ such that $\{ C_0 , \dots , C_d \}$
is an $E$-Lebesgue cover of $X$. 
By Lemma~\ref{L-lower dim 2}, for every $j=0,\dots , d$ there are a collection of 
open towers $\{ (V_i , S_i) \}_{i\in I_j}$ with chromatic number at most $\dim (X)+1$
and levels of diameter less than $\delta$ and sets $O_i \subseteq V_i$ such that 
$C_j\subseteq\bigcup_{j=0}^d \bigsqcup_{i\in I_j} S_i O_i \subseteq \bigcup_{j=0}^d \bigsqcup_{i\in I_j} S_i V_i \subseteq U_j$ 
and 
\begin{align}\label{E-inclusion}
[x]_{R_{C_j,E}} \subseteq S_i x
\end{align}
for every $i\in I_j$, $t\in S_i$, and $x\in tO_i \cap C_j$. 
Note that the collection of open towers $\{ (V_i , S_i ) \}_{i\in I_j ,\, 0\leq j\leq d}$
has chromatic number at most $(d+1)(\dim (X)+1)$.

Now let $x\in X$. Since the cover $\{ C_0 , \dots , C_d \}$ is $E$-Lebesgue, 
there is a $0\leq j\leq d$ such that 
$Ex\subseteq C_j$. Since $x\in C_j$ there are $i\in I_j$, $t\in S_i$, 
and $y\in O_i$ such that $x=ty$.
By (\ref{E-inclusion}) and our choice of $j$, for every $s\in E$ we have 
$sty \in [ty]_{R_{C_j ,E}} \subseteq S_i y$
so that $Ety \subseteq S_i y$ and hence $Et\subseteq S_i$.
This shows that the collection of open towers $\{ (V_i , S_i ) \}_{i\in I_j ,\, 0\leq j\leq d}$
is $E$-Lebesgue.
We have thus verified that $\ftowdimplus (X,G) \leq \dadplus (X,G)\cdot\dimplus (X)$,
as desired.
\end{proof}

\begin{corollary}\label{C-dimensions Cantor}
Suppose that $X$ is zero-dimensional. Then the action $G\curvearrowright X$ satisfies
\begin{align*}
\towdim (X,G) 
= \ftowdim (X,G)
= \dad (X,G) 
= \amdim (X,G) .
\end{align*}
\end{corollary}

\section{Tower dimension and nuclear dimension}\label{S-towdim nucdim}

Let $G\curvearrowright X$ be a free action on a compact Hausdorff space.
We write $C(X)\rtimes_\lambda G$ for the associated reduced crossed product.
In Section~8 of \cite{GueWilYu17}, Guentner, Willett, and Yu showed that
\[
\nucdimplus (C(X)\rtimes_\lambda G) \leq \dadplus (X,G)\cdot\dimplus (X) . 
\]
By Theorem~\ref{T-dimensions} this implies that 
\begin{align}\label{E-nucdim}
\nucdimplus (C(X)\rtimes_\lambda G)\leq \towdimplus (X,G) \cdot \dimplus (X) .
\end{align}
We will give here a shorter direct proof of (\ref{E-nucdim}) 
in order to illustrate the formal affinity between tower dimension and nuclear dimension.
This can be seen as a distillation
of the arguments in Section~8 of \cite{GueWilYu17} into their simplest combinatorial form.

First we recall the definition of nuclear dimension \cite{WinZac10}.
We use the abbreviation {\it c.p.c.}\ for ``completely positive contractive''.
A map $\varphi : A\to B$ between C$^*$-algebras is {\it order zero} if
it preserve orthogonality, that is,
$\varphi (a_1 )\varphi (a_2 ) = 0$ for all $a_1 , a_2 \in A$ satisfying $a_1 a_2 = 0$.

\begin{definition}
The {\it nuclear dimension} $\nucdim (A)$ of a C$^*$-algebra $A$ is the least integer $d\geq 0$ 
such that for every finite
set $\Omega\subseteq A$ and $\eps > 0$ there are finite-dimensional C$^*$-algebras
$B_0 , \dots , B_d$ and linear maps 
\begin{align*}
A\stackrel{\varphi}{\longrightarrow} B_0 \oplus\cdots\oplus B_d \stackrel{\psi}{\longrightarrow} A
\end{align*}
such that $\varphi$ is c.p.c., $\psi |_{B_i}$ is c.p.c.\ and order zero for each $i=0,\dots , d$,
and 
\[
\| \psi\circ\varphi (a) - a \| < \eps
\]
for every $a\in\Omega$.
If no such $d$ exists then we set $\nucdim (A) = \infty$.
\end{definition}

Let $(V,S)$ be an open tower.
Write $A_{V,S}$ for the $C^*$-subalgebra of $C(X)\rtimes_\lambda G$ generated by the sets
$u_s C_0 (V) u_t^*$ for $s,t\in S$. Denoting by $M_T$ the matrix algebra with entries indexed 
by pairs in $T\times T$ and by $\{ e_{s,t} \}_{s,t\in T}$ the matrix units of $M_T$,
there is a canonical isomorphism $M_T \otimes C_0 (V) \to A_{V,S}$
determined by
\[
e_{s,t} \otimes f \mapsto u_s f u_t^* 
\]
for $s,t\in S$ and $f\in C_0 (V)$.

\begin{theorem}\label{T-nucdim towdim}
The action $G\curvearrowright X$ satisfies
\[
\nucdimplus (C(X)\rtimes_\lambda G) \leq \towdimplus (X,G) \cdot \dimplus (X) .
\]
\end{theorem}

\begin{proof}
We denote the induced action of $G$ on $C(X)$ by $\alpha$,
that is, $\alpha_s (f)(x) = f(s^{-1} x)$ for all $s\in G$, $f\in C(X)$, and $x\in X$.

We may assume that $\towdim (X,G)$ is finite. For brevity we denote this quantity by $d$.
Let $\Omega$ be a finite subset of $C(X)\rtimes G$ and $\eps > 0$. In order to 
verify the existence of the desired maps in the definition of nuclear dimension
which approximately factorize the identity map on $C(X)\rtimes G$ to within $\eps$ 
on the set $\Omega$, we may assume that $\Omega = \{ fu_s : f\in\Upsilon , \, s\in F \}$
where $\Upsilon$ is a finite set of functions  
in $C(X)$ and $F$ is a finite subset of $G$ satisfying $F^{-1} = F$ 
and $e\in F$.

Let $n$ be an integer greater than $1$, to be determined.
By the definition of tower dimension, there is an $F^n$-Lebesgue 
collection of open towers $\{ (V_i ,S_i) \}_{i\in I}$ covering $X$
such that the family $\{ S_i V_i \}_{i\in I}$ has chromatic number at most $d+1$.
For convenience we may assume that for each $i$ the set $S_i$ contains $e$,
for if necessary we can choose a $t \in S_i$ and replace $S_i$ by $S_i t^{-1}$ 
and $V_i$ by $tV_i$. 

Take a partition of unity $\{ \hat{g}_{i,t} \}_{i\in I, t\in S_i}$ subordinate to the
open cover $\{ tV_i \}_{i\in I,t\in S_i}$. We will modify these functions so that over each
tower they are dynamically generated by a single function and together sum to a function which,
while no longer necessarily equal to $1$, is still greater than or equal to $1$. 
For every $i\in I$ set $g_i = \max_{t\in S_i} \alpha_{t^{-1}} (\hat{g}_{i,t} )$.
Then for every $i\in I$ and $t\in S_i$ the support of $\alpha_t (g_i )$ is contained in $tV_i$,
and $\alpha_t (g_i ) \geq \hat{g}_{i,t}$, which implies that 
$\sum_{i\in I} \sum_{t\in S_i} \alpha_t (g_i ) \geq 1$.

Let $i\in I$. Set $B_{i,n} = \bigcap_{t\in F^n} tS_i$ and $B_{i,0} = S_i \setminus \bigcap_{t\in F} tS_i$.
For $k=1,\dots ,n-1$ set 
\[
B_{i,k} = \bigg(\bigcap_{t\in F^k} tS_i \bigg)\setminus\bigcap_{t\in F^{k+1}} tS_i . 
\]
The sets $B_{i,k}$ for $k=0,\dots ,n$ form a partition of $S_i$, and for all $s\in F$ we have
\begin{enumerate}
\item $sB_{i,k} \subseteq B_{i,k-1} \cup B_{i,k} \cup B_{i,k+1}$ for every $k=1,\dots ,n-1$, 

\item $sB_{i,n} \subseteq B_{i,n-1} \cup B_{i,n}$.
\end{enumerate}
Since for each $t\in S_i$ the function $\alpha_t (g_i )$ is supported in the tower level $tV_i$, it follows that
the function
\[
\hat{h}_i = \sum_{k=0}^n \sum_{t\in B_{i,k}} \frac{k}{n} \alpha_t (g_i ) 
\]
satisfies 
$\sup_{x\in X} |\hat{h}_i (s^{-1} x) - \hat{h}_i (x)|\leq 1/n$ 
for every $s\in F$.
Put $H = \sum_{i\in I} \hat{h}_i$. Then $H\geq 1$ by the $F^n$-Lebesgue condition, 
and so for every $i$ we can set $h_i = H^{-1} \hat{h}_i$, which gives us a partition of unity
$\{ h_i \}_{i\in I}$ in $C(X)$.

Let $s\in F$.
Let $x\in X$.
The collection of all $i\in I$ such that $x\in S_i V_i$
has cardinality at most $d+1$, and so the difference between the values of 
$H$ at $x$ and $s^{-1} x$ is at most $(d+1)/n$.
Since $H \geq 1$, it follows that the difference between the values
of $H^{-1}$ at $x$ and $s^{-1} x$ is also at most $(d+1)/n$. We then get, for every $i$,
\begin{align}\label{D-bound}
\| u_s h_i - h_i u_s \|
&= \| u_s h_i u_s^{-1} - h_i \| \\
&= \sup_{x\in X} |h_i (s^{-1} x) - h_i (x)| \notag \\
&\leq \sup_{x\in X} H(s^{-1} x)^{-1} \big| \hat{h}_i (s^{-1} x) - \hat{h}_i (x)\big| \notag \\
&\hspace*{30mm} \ + \sup_{x\in X} \big| H(s^{-1} x)^{-1} - H(x)^{-1}\big| \hat{h} (x) \notag \\
&\leq \frac{d+2}{n} . \notag 
\end{align}

Since the collection $\{ S_i V_i \}_{i\in I}$ has chromatic number at most $d+1$,
there is a partition
$I_0 , \dots , I_d$ of $I$  such that for every $k=0,\dots, d$
the collection $\{ S_i V_i \}_{i\in I_k}$ is disjoint.
For each $k=0,\dots ,d$ set $q_k = \sum_{i\in I_k} h_i$.

For $i\in I$ we write $A_i$ for the $C^*$-subalgebra of $C(X)\rtimes G$ generated by the sets
$u_s C_0 (V_i ) u_t^*$ for $s,t\in S_i$. 
Let $k\in \{ 0,\dots , d\}$ and set $A_k = \bigoplus_{i\in I_k} A_i$.
Since the $A_i$ for $i\in I_k$ are pairwise orthogonal
as sub-C$^*$-algebras of $C(X)\rtimes G$, we can view $A_k$ as a C$^*$-subalgebra of $C(X)\rtimes G$.
Since $A_i \cong M_{S_i} \otimes C_0 (V_i)$ for every $i$ (as explained prior to the statement of the theorem), 
the nuclear dimension of $A_i$ is at most $\dim (X)$, as one can verify by a straightforward partition of
unity argument using the formulation of covering dimension in term of the chromatic numbers of open covers
(see the proof of Proposition~3.4 in \cite{KirWin04}).
Noting that $fu_s q_k = u_s \alpha_s (f)q_k \in A_k$ for every $f\in\Upsilon$ and $s\in F$ 
since $\hat{h}_i$ vanishes on $B_{i,0}$ for each $i$,
we can thus find finite-dimensional $C^*$-algebras $D_{k,0} , \dots , D_{k,m_k}$ with $m_k \leq \dim (X)$,
a c.p.c.\ map $\theta_k : A_k \to D_{k,0} \oplus\cdots\oplus D_{k,m_k}$, 
and a map $\psi_k : D_{k,0} \oplus\cdots\oplus D_{k,m_k} \to A_k \subseteq C(X)\rtimes G$ 
whose restriction to each summand is c.p.c.\ and order zero 
such that 
\begin{align}\label{E-homogeneous}
\| \psi_k \circ\theta_k (fu_s q_k ) - fu_s q_k \| < \frac{\eps}{2(d+1)}
\end{align}
for all $f\in\Upsilon$ and $s\in F$. 
By Arveson's extension theorem we can extend $\theta_k$ to a c.p.c.\ map
$C(X)\rtimes G \to D_{k,0} \oplus\cdots\oplus D_{k,m_k}$, which we will again call $\theta_k$.
Define the c.p.c.\ map
$\varphi_k : C(X)\rtimes G \to D_{k,0} \oplus\cdots\oplus D_{k,m_k}$ by 
\[
\varphi_k (a) = \theta_k (q_k^{1/2} a q_k^{1/2} ) .
\]

Now define the maps
\begin{align*}
C(X)\rtimes G \stackrel{\varphi}{\longrightarrow} 
\bigoplus_{k=0}^d D_{k,0} \oplus\cdots\oplus D_{k,m_k} 
\stackrel{\psi}{\longrightarrow} C(X)\rtimes G
\end{align*}
by $\varphi = \bigoplus_{k=0}^d \varphi_k$ and
\[
\psi (a_0 , \dots , a_d ) = \psi_0 (a_0 ) + \dots + \psi_d (a_d ) . 
\]
Then $\varphi$ is c.p.c.\ and the restriction of $\psi$ to each $D_{k,j}$ is c.p.c.\ and order zero.
Since $m_0 + \cdots + m_d \leq \dimplus (X,G) \cdot\dimplus (X)$,
to obtain the desired upper bound on $\nucdim (C(X)\rtimes G)$ it remains to verify that 
$\| \psi\circ\varphi (fu_s ) - fu_s \| < \eps$ for all $f\in\Upsilon$ and $s\in F$. 

By a straightforward functional calculus argument that uses a polynomial approximation 
to the function $x\mapsto x^{1/2}$ on $[0,1]$, we see from (\ref{D-bound}) 
that if $n$ is small enough relative to $d$ then
for each $i\in I$, $f\in\Upsilon$, and $s\in F$ we will have
\[
\| h_i^{1/2} fu_s h_i^{1/2} - fu_s h_i \| < \frac{\eps}{2(d+1)} .
\]
Let $s\in F$ and $k\in \{0,\dots ,d\}$. Since for every $i\in I_k$ the element
$h_i^{1/2} fu_s h_i^{1/2} - fu_s h_i$
belongs to $A_i$ and the sub-C$^*$-subalgebras $A_i$ for $i\in I_k$ are pairwise orthogonal, we get
\[
\| q_k^{1/2} fu_s q_k^{1/2} - fu_s q_k \| = 
\max_{i\in I_k} \, \| h_i^{1/2} fu_s h_i^{1/2} - fu_s h_i \| < \frac{\eps}{2(d+1)} .
\]
Using (\ref{E-homogeneous}) this yields
\begin{align*}
\| \psi_k \circ\varphi_k (fu_s ) - fu_s q_k \|
&\leq \| \psi_k \circ\theta_k (q_k^{1/2} fu_s q_k^{1/2} - fu_s q_k ) \| \\
&\hspace*{20mm} \ + \| \psi_k \circ\theta_k (fu_s q_k ) - fu_s q_k \| \\
&< \| h_k^{1/2} fu_s q_k^{1/2} - fu_s q_k \| + \frac{\eps}{2(d+1)} < \frac{\eps}{d+1} ,
\end{align*}
whence
\begin{align*}
\| \psi\circ\varphi (fu_s ) - fu_s \| 
&= \bigg\| \sum_{k=0}^d (\psi_k \circ\varphi_k (fu_s ) - fu_s q_k ) \bigg\| \\
&\leq \sum_{k=0}^d \| \psi_k \circ\varphi_k (fu_s ) - fu_s q_k \| \\
&< (d+1)\cdot \frac{\eps}{d+1} = \eps ,
\end{align*}
as desired.
\end{proof}

\section{Tower dimension and comparison}\label{S-towdim comparison}

We aim here to establish Theorem~\ref{T-tower}.

\begin{lemma}\label{L-Folner towers}
Suppose that $G$ is amenable.
Let $G\curvearrowright X$ be a free action with tower dimension $d<\infty$.
Let $K$ be a finite subset of $G$ and $\delta > 0$. Then 
there is a finite collection $\{ (V_i,S_i ) \}_{i\in I}$ of open towers covering $X$
such that $S_i$ is $(K,\delta )$-invariant for every $i\in I$ 
and the family $\{ S_i V_i \}_{i\in I}$ has chromatic number at most $d+1$.
\end{lemma}

\begin{proof}
By the main theorem of \cite{DowHucZha16} there exist nonempty
$(K,\delta )$-invariant finite sets $F_1 , \dots , F_n \subseteq G$
and sets $C_1 , \dots , C_n \subseteq G$ such that
\begin{align*}
G = \bigsqcup_{k=1}^n \bigsqcup_{c\in C_k } F_k c .
\end{align*}
Set $F = F_1 F_1^{-1} \cup\cdots\cup F_n F_n^{-1}$. By our tower dimension hypothesis there is 
a finite $F$-Lebesgue collection of open towers $\{ (V_i,T_i ) \}_{i\in I}$ covering $X$
such that the family $\{ S_i V_i \}_{i\in I}$ has chromatic number at most $d+1$.
For each $i\in I$ set
\begin{gather*}
T_i' = \bigcup \{ F_k c : 1\leq k\leq n,\, c\in C_k , \text{ and } F_k c \subseteq T_i \} , \\
T_i'' = \bigcap_{s\in F} s^{-1} T_i .
\end{gather*}
Let $i\in I$. Let $x\in T_i \setminus T_i'$.
Take $1\leq k\leq n$ and $c\in C_k$ such that $x\in F_k c \cap T_i$. Then 
there exists a $y\in F_k c \cap (G\setminus T_i )$. We have $x=sc$ and $y=tc$ for some
$s,t\in F_k$, whence $ts^{-1} x = y \notin T_i$, which shows that $x\notin T_i''$ since $ts^{-1} \in F$.
We conclude from this that $T_i'' \subseteq T_i'$.
It follows by the $F$-Lebesgue condition that the towers $(V_i,T_i' )$ for $i\in I$ cover $X$.

Finally, for each $i\in I$ write $T_i' = \bigsqcup_{j\in J_i} S_j$ where
each $S_j$ has the form $F_k c$ for some $1\leq k\leq n$ and $c\in C_k$. Then
the collection of open towers $\{ (V_i,S_j ) \}_{i\in I, j\in J_i}$ covers $X$,
each of its shapes is $(K,\delta )$-invariant, and the family $\{ S_i V_i \}_{i\in I}$ 
has chromatic number at most $d+1$, as desired.
\end{proof}

\begin{theorem}\label{T-tower} 
Suppose that $G$ is amenable.
Let $X$ be a compact metric space with covering dimension $c<\infty$. 
Let $G\curvearrowright X$ be a free action with tower dimension $d<\infty$. 
Then the action has $((c+1)(d+1)-1)$-comparison. 
\end{theorem}

\begin{proof}
Let $A$ be a closed subset of $X$ and $B$ an open subset of $X$ such that 
$\mu (A) < \mu (B)$ for all $\mu\in M_G (X)$. By Lemma~\ref{L-portmanteau} we can find
an $\eta > 0$ such that the sets
\begin{align*}
B_- &= \{ x\in X : d(x,X\setminus B ) > \eta \} , \\
A_+ &= \{ x\in X : d(x,A) \leq \eta \}
\end{align*}
satisfy $\mu (A_+ ) + \eta \leq \mu (B_- )$ for all $\mu\in M_G (X)$.

We claim that there are a finite set $K\subseteq G$ and a $\delta > 0$ such
that if $F$ is a nonempty $(K,\delta )$-invariant finite subset of $G$ then 
for all $x\in X$ we have
\begin{align}\label{E-distribution}
\frac{1}{|F|} \sum_{s\in F} \unit_{A_+} (sx) + \frac{\eta}{2} 
&\leq \frac{1}{|F|} \sum_{s\in F} \unit_{B_-} (sx) .
\end{align}
Suppose that this is not possible.
Then there exists a F{\o}lner sequence $\{ F_n \}$ and a sequence $\{ x_n \}$ in $X$ such that,
writing $\mu_n$ for the probability measure $(1/|F_n|) \sum_{s\in F_n} \delta_{sx}$, we have
\begin{align*}
\mu_n (A_+ ) + \frac{\eta}{2} > \mu_n (B_- )
\end{align*}
for all $n$.
By passing to a subsequence
we may assume that the sequence $\{ \mu_n \}$ converges to some $\mu\in M(X)$,
and the F{\o}lner property implies that $\mu$ is $G$-invariant, as is easily verified.
Since $B_-$ is open and $A_+$ is closed, the portmanteau theorem yields
\begin{align*}
\mu (B_- ) + \frac{\eta}{2} 
\leq \liminf_{n\to\infty} \mu_n (B_- ) + \frac{\eta}{2} 
\leq \limsup_{n\to\infty} \mu_n (A_+ ) + \eta 
\leq \mu (A_+ ) + \eta ,
\end{align*}
contradicting our choice of $\eta$. The desired $K$ and $\delta$ thus exist.

By Lemma~\ref{L-Folner towers} there are
a finite collection $\{ (V_i,T_i ) \}_{i\in I}$ of open towers covering $X$
and a partition $I = I_0 \sqcup\cdots\sqcup I_d$ such that for every $i\in I$ the shape
$T_i$ is $(K,\delta )$-invariant and for every $j=0,\dots ,d$
the sets $T_i V_i$ for $i\in I_j$ are pairwise disjoint.
By normality we can find for every $i\in I$ a closed set
$V_i' \subseteq V_i$ such that the sets $T_i V_i'$ for $i\in I$ still cover $X$.
Using the formulation of covering dimension in term of the chromatic numbers of open covers,
we can then find, for each $i\in I$, a finite collection $\cU_i$
of open subsets of $V_i$ such that
\begin{enumerate}
\item the collection $\cU_i$ covers $V_i'$,

\item each of the sets $sU$ for $s\in T_i$ and $U\in\cU_i$ has diameter less than $\eta$, and

\item there is a partition $\cU_i = \cU_{i,0} \sqcup\cdots\sqcup \cU_{i,c}$
such that the collection $\cU_{i,j}$ is disjoint for each $j$.
\end{enumerate}
For convenience we reindex the collection of towers 
$\{ (U , T_i ) \}_{i\in I, U\in\cU_i}$ as $\{ (U_j , S_j ) \}_{j\in J}$.
Then the shapes $S_j$ are all $(K,\delta )$-invariant
and, setting $m = (c+1)(d+1)-1$, there is a partition 
$J = J_0 \sqcup\cdots\sqcup J_m$ such that for each $k=0,\dots ,m$
the sets $S_j U_j$ for $j\in J_k$ are pairwise disjoint.

Let $0\leq k\leq m$ and $j\in J_k$. 
Since the levels of the tower $(U_j , S_j )$ all have diameter less than $\eta$,
if $sU_j \cap A \neq \emptyset$ for some $s\in S_j$ then $sU_j \subseteq A_+$,
and so by (\ref{E-distribution}) the sets 
\begin{align*}
S_{j,1} &= \{ s\in S_j : sU_j \cap A \neq \emptyset \} , \\
S_{j,2} &= \{ s\in S_j : sU_j \cap B_- \neq \emptyset \}
\end{align*}
must satisfy $|S_{j,1} |/|S_j | + \eta /2 \leq |S_{j,2} |/|S_j |$ and hence 
$|S_{j,1} |\leq |S_{j,2} |$. We can thus find an injection $\varphi_j : S_{j,1} \to S_{j,2}$.
Now the sets $sU_j$ for $s\in S_{j,1}$ and $j\in J$ cover $A$, while
for each $k=0,\dots ,m$ the pairwise disjoint sets $\varphi (s) U_j = (\varphi (s)s^{-1} )sU_j$ 
for $j\in J_k$ and $s\in S_{j,1}$ are contained in $B$
since the levels of the tower $(U_j , S_j )$ all have diameter less than $\eta$.
This verifies that $A\prec_m B$, as desired.
\end{proof}

\section{Almost finiteness}\label{S-af}

We begin by recalling the following notion of castle from Definition~\ref{D-castle 1}.

\begin{definition}\label{D-castle 2}
Let $G\curvearrowright X$ be a free action on a compact metric space. 
A {\it castle} is a finite collection of towers $\{ (V_i , S_i) \}_{i\in I}$
such that the sets $S_i V_i$ for $i\in I$ are pairwise disjoint.
The {\it levels} of the castle are the sets $sV_i$ for $i\in I$ and $s\in S_i$.
We say that the castle is {\it open} if each of the towers is open, 
and {\it clopen} if each of the towers is clopen.
\end{definition}

\begin{definition}\label{D-af}
We say that a free action $G\curvearrowright X$ on a compact metric space
is {\it almost finite} if for every $n\in\Nb$, finite set $K\subseteq G$, and $\delta > 0$
there are 
\begin{enumerate}
\item an open castle $\{ (V_i ,S_i ) \}_{i\in I}$ whose shapes are $(K,\delta )$-invariant
and whose levels have diameter less than $\delta$, 

\item sets $S_i' \subseteq S_i$ such that $|S_i' | < |S_i |/n$ and
\[
X\setminus \bigsqcup_{i\in I} S_i V_i \prec \bigsqcup_{i\in I} S_i' V_i .
\]
\end{enumerate}
\end{definition}

\begin{remark}\label{R-tower comparison}
Observe in the context of Definition~\ref{D-af} 
that if we have sets $S_i' \subseteq S_i$ satisfying
\[
X\setminus \bigsqcup_{i\in I} S_i V_i \prec \bigsqcup_{i\in I} S_i' V_i 
\]
then any other sets $S_i'' \subseteq S_i$ with $|S_i'' | \geq |S_i' |$ will similarly satisfy
\[
X\setminus \bigsqcup_{i\in I} S_i V_i \prec \bigsqcup_{i\in I} S_i'' V_i .
\]
since the relation $\prec$ is transitive and 
$\bigsqcup_{i\in I} S_i' V_i \prec\bigsqcup_{i\in I} S_i'' V_i$.
The latter follows from the fact that for each $i$ we have $S_i' V_i \prec S_i'' V_i$, 
which can be witnessed by taking
an injection $\varphi : S_i' \to S_i''$ and considering the open collections
$\{ sV_i : s\in S_i' \}$ and $\{ \varphi (s)V_i : s\in S_i' \}$, the first of which
partitions $S_i' V_i$ and the second of which partitions the subset $\varphi (S_i')V_i$
of $S_i'' V_i$.
\end{remark}

\begin{remark}\label{R-af extn}
Almost finiteness does not pass to extensions, the obstruction being the diameter condition.
For example, the minimal actions in \cite{GioKer10} factor onto an odometer,
which is almost finite, but are not themselves almost finite by Theorem~\ref{T-af Z-stable},
since their crossed product fails to be $\cZ$-stable. 
See however Theorem~\ref{T-af extension finite}.
\end{remark}

\begin{example}
Every free $\Zb^m$-action on a zero-dimensional compact metrizable space is almost finite.
This was established in Lemma~6.3 of \cite{Mat12} in the language of groupoids, whose translation
to Definition~\ref{D-af} is discussed in the first paragraph of Section~\ref{S-zero dim}.
\end{example}

The following was shown in \cite{ConJacKerMarSewTuc17}.

\begin{theorem}\label{T-tiling}
Let $G$ be a countable amenable group. Then a generic free minimal action 
of $G$ on the Cantor set is almost finite.
\end{theorem}

The following two facts are simple consequences of Definition~\ref{D-af}.

\begin{proposition}\label{P-af inverse}
Almost finiteness is preserved under inverse limits of free actions.
\end{proposition}

\begin{proposition}
Let $G\curvearrowright X$ be a free action on a compact metrizable space, and
suppose that $G$ can be expressed as a union of an increasing sequence $G_1 \subseteq G_2 \subseteq\dots$
of subgroups such that the restriction action $G_n \curvearrowright X$ is almost finite
for every $n$. Then the action $G\curvearrowright X$ is almost finite.
\end{proposition}

\begin{problem}
Let $G\curvearrowright X$ be a uniquely ergodic
free minimal action of a countable amenable group on the Cantor set. Must it be almost finite?
\end{problem}

\section{Almost finiteness and comparison}\label{S-af comparison}

In Theorem~\ref{T-comparison af} we relate almost finiteness and comparison.
By combining this with Theorem~\ref{T-tower} we are then able to give a connection 
between tower dimension, almost finiteness, and comparison, which we record
as Theorem~\ref{T-comparison af tower}.

\begin{lemma}\label{L-regularity}
Let $X$ be a compact metrizable space and let $\Omega$ be a weak$^*$ closed subset of $M(X)$.
Let $A$ be a closed subset of $X$ such that $\mu (A) = 0$ for all $\mu\in\Omega$, and let $\eps > 0$.
Then there is a $\delta > 0$ such that 
\[
\mu (\{ x\in X : d(x,A) \leq \delta \} ) < \eps
\]
for all $\mu\in\Omega$.
\end{lemma}

\begin{proof}
Suppose that the conclusion does not hold. 
Then for every $n\in\Nb$ we can find a $\mu_n \in\Omega$
such that the set $A_n = \{ x\in X : d(x,A) \leq 1/n \}$ satisfies $\mu_n (A_n ) \geq\eps$.
By the compactness of $\Omega$ there is a subsequence $\{ \mu_{n_k} \}$ of $\{ \mu_n \}$ 
which weak$^*$ converges to some $\mu\in\Omega$.
For a fixed $j\in\Nb$ we have $\mu_{n_k} (A_{n_j} ) \geq\eps$ for every $k \geq j$,
and since $A_{n_j}$ is closed the portmanteau theorem then yields
\begin{align*}
\mu (A_{n_j} ) \geq \limsup_{k\to\infty} \mu_{n_k} (A_{n_j} ) \geq \eps .
\end{align*}
As $A$ is closed it is equal to the intersection of the decreasing sequence of sets $A_{n_j}$,
and so $\mu (A) = \lim_{j\to\infty} \mu (A_{n_j} ) \geq\eps$,
in contradiction to our hypothesis.
\end{proof}

\begin{theorem}\label{T-comparison af}
Suppose that $G$ is amenable.
Let $G\curvearrowright X$ be a free minimal action and consider the following conditions:
\begin{enumerate}
\item the action is almost finite,

\item the action has comparison,

\item the action has $m$-comparison for all $m\geq 0$,

\item the action has $m$-comparison for some $m\geq 0$.
\end{enumerate}
Then (i)$\Rightarrow$(ii)$\Rightarrow$(iii)$\Rightarrow$(iv),
and if $E_G (X)$ is finite then all four conditions are equivalent.
\end{theorem}

\begin{proof}
(i)$\Rightarrow$(ii). Let $A$ be a closed subset of $X$
and $B$ an open subset of $X$ such that
$\mu (A) < \mu (B)$ for all $\mu\in M_G (X)$. 
We aim to show that $A\prec B$, which will establish (ii).
By Lemma~\ref{L-portmanteau} there exists an $\eta > 0$ such that 
$\mu (A) + \eta \leq \mu (B)$ for all $\mu\in M_G (X)$.
As the set $B\setminus A$ must be nonempty, we can pick a $y\in B\setminus A$.
By Lemma~\ref{L-regularity} there is a $\kappa > 0$ such that the
closed ball $C = \{ x\in X : d(x,y) \leq \kappa \}$ is contained in $B\setminus A$
and satisfies $\mu (C) \leq \eta /2$ for all $\mu\in M_G (X)$.
By minimality the open ball $C_- = \{ x\in X : d(x,y) < \kappa /2 \}$
satisfies $\mu (C_- ) > 0$ for all $\mu\in M_G (X)$, and so by Lemma~\ref{L-portmanteau}
(taking $A=\emptyset$ and $B = C_-$ there) there is a $\theta > 0$ such that
$\mu (C_- ) \geq\theta$ for all $\mu\in M_G (X)$.

Set $\tilde{B} = B\setminus C$. Then for all $\mu\in M_G (X)$ we have
\begin{align*}
\mu (\tilde{B} ) = \mu (B) - \mu (C) \geq \mu (A) + \frac{\eta}{2} > \mu (A)
\end{align*}
and so by Lemma~\ref{L-portmanteau} there exists an $\eta' > 0$ with $\eta' \leq \eta$
such that the sets
\begin{align*}
B_- &= \{ x\in X : d(x,X\setminus \tilde{B} ) > \eta' \} , \\
A_+ &= \{ x\in X : d(x,A) \leq \eta' \}
\end{align*}
satisfy $\mu (A_+ ) + \eta' \leq \mu (B_- )$ for all $\mu\in M_G (X)$.
Note that each of the sets $A_+$ and $B_-$ is disjoint from $C$.

We claim that there are a finite set $K\subseteq G$ and a $\delta > 0$ such
that if $F$ is a nonempty $(K,\delta )$-invariant finite subset of $G$ then 
for all $x\in X$ the following both hold:
\begin{align}\label{E-uniform}
\frac{1}{|F|} \sum_{s\in F} \unit_{A_+} (sx) + \frac{\eta'}{2} 
&\leq \frac{1}{|F|} \sum_{s\in F} \unit_{B_-} (sx) , \\
\frac{1}{|F|} \sum_{s\in F} \unit_{C_-} (sx) &\geq \frac{\theta}{2} . \label{E-uniform 2}
\end{align}
Suppose to the contrary that this is not possible.
Then we can find a F{\o}lner sequence $\{ F_n \}$ and a sequence $\{ x_n \}$ in $X$ such that,
writing $\mu_n$ for the probability measure $(1/|F_n|) \sum_{s\in F_n} \delta_{sx}$, 
one of the following holds:
\begin{enumerate}
\item $\mu_n (A_+ ) + \eta' /2 > \mu_n (B_- )$ for all $n$, 

\item $\mu_n (C_- ) < \theta /2$ for all $n$. 
\end{enumerate}
Suppose first that (i) holds. By passing to a subsequence
we may assume that the sequence $\{ \mu_n \}$ converges to some $\mu\in M(X)$,
and it is readily verified using the F{\o}lner property that $\mu$ is $G$-invariant.
Since $B_-$ is open and $A_+$ is closed, the portmanteau theorem then yields
\begin{align*}
\mu (B_- ) + \frac{\eta'}{2} 
\leq \liminf_{n\to\infty} \mu_n (B_- ) + \frac{\eta'}{2} 
\leq \limsup_{n\to\infty} \mu_n (A_+ ) + \eta' 
\leq \mu (A_+ ) + \eta' ,
\end{align*}
a contradiction. If on the other hand (ii) holds,
then as before we may assume that $\{ \mu_n \}$ converges to some $\mu\in M_G (X)$,
and since $C_-$ is open the portmanteau theorem yields
\begin{align*}
\mu (C_- ) \leq \liminf_{n\to\infty} \mu_n (C_- ) \leq \frac{\theta}{2} ,
\end{align*}
a contradiction. We may thus find the desired $K$ and $\delta$. 

Set $\eps = \min \{ \eta' , \kappa /2 \}$
and choose an integer $n > 3/\theta$. Then by almost finiteness there are 
\begin{enumerate}
\item an open castle $\{ (V_i,S_i ) \}_{i\in I}$ whose shapes are $(K,\delta )$-invariant
and whose levels have diameter less than $\eps$, and

\item sets $S_i' \subseteq S_i$ such that $|S_i' | < |S_i |/n$ and the set 
$D := \bigsqcup_{i\in I} S_i V_i$ satisfies
\[
X\setminus D \prec \bigsqcup_{i\in I} S_i' V_i .
\]
\end{enumerate}
Let $i\in I$. Since the levels of the towers all have diameter less than 
both $\eta'$ and $\kappa /2$, by (\ref{E-uniform}) and (\ref{E-uniform 2}) the sets 
\begin{align*}
S_{i,1} &= \{ s\in S_i : sV_i \cap A \neq\emptyset \} , \\
S_{i,2} &= \{ s\in S_i : sV_i \cap B_- \neq\emptyset \} , \\
S_{i,3} &= \{ s\in S_i : sV_i \cap C_- \neq\emptyset \}
\end{align*}
satisfy 
\begin{align*}
\frac{|S_{i,1} |}{|S_i|} + \frac{\eta'}{2} \leq \frac{|S_{i,2} |}{|S_i|} 
\hspace*{5mm}\text{and}\hspace*{5mm}
\frac{|S_{i,3} |}{|S_i|} \geq \frac{\theta}{2} 
\end{align*}
so that $|S_{i,1} | \leq |S_{i,2} |$ and $|S_{i,3} | \geq |S_i' |$. 
We can thus find injective maps $\varphi_i : S_{i,1} \to S_{i,2}$
and $\psi_i : S_i' \to S_{i,3}$.

Since $X\setminus D \prec \bigsqcup_{i\in I} S_i' V_i$ we can find a finite collection $\cU$ of
open subsets of $X$ which cover $X\setminus D$ and a $t_U \in G$ for each $U\in\cU$
such that the images $t_U U$ for $U\in\cU$ are pairwise disjoint subsets of $\bigsqcup_{i\in I} S_i' V_i$.
For all $U\in\cU$, $i\in I$, and $s\in S_i'$ write $W_{U,i,s}$ for the (possibly empty) open set
$U\cap t_U^{-1} sV_i$. These open sets cover $X\setminus D$, 
and so in particular cover $(X\setminus D)\cap A$, and the images
$\psi_i (t_U ) W_{U,i,s}$ for $U\in\cU$, $i\in I$, and $s\in S_i'$ are pairwise disjoint subsets of
$B\cap\bigsqcup_{i\in I} S_{i,3} V_i$.
At the same time, the open sets $sV_i$ for $i\in I$ and $s\in S_{i,1}$ cover $D\cap A$,
while the images $\varphi_i (s)V_i = (\varphi_i (s)s^{-1} )sV_i$ for $i\in I$ and $s\in S_{i,1}$ 
are pairwise disjoint subsets of $B \cap \bigsqcup_{i\in I} S_{i,2} V_i$.
Since the sets $\bigsqcup_{i\in I} S_{i,3} V_i$ and $\bigsqcup_{i\in I} S_{i,2} V_i$ are disjoint,
we have thus verified that $A\prec B$.

(ii)$\Rightarrow$(iii)$\Rightarrow$(iv). Trivial.

Now suppose that $E_G (X)$ is finite and let us verify (iv)$\Rightarrow$(i).
We thus suppose that there is an $m\in\Nb$ such that the action has $m$-comparison.
We may assume that $G$ is infinite, for otherwise minimality implies that $X$ 
consists of a single orbit, in which case the action is obviously almost finite.
Write $E_G (X) = \{ \mu_1 , \dots , \mu_q \}$ and set $\mu = (1/q) \sum_{k=1}^q \mu_k \in M_G (X)$.
Let $K$ be a finite subset of $G$, $\delta > 0$, and $n\in\Nb$. 
Put $\eps = 1/(4nq(m+1))$. Choose an integer $N > 1/\eps$.
Since $G$ is infinite and $m\cdot 4q\eps < 1$
we can find a finite set $K' \subseteq G$ with $K\subseteq K$
and a $\delta' > 0$ with $\delta' \leq \delta$ such that every nonempty
$(K' ,\delta' )$-invariant finite set $F\subseteq G$ has large enough cardinality
so that it has $m$ pairwise disjoint subsets of equal cardinality $\kappa$ 
satisfying $2q\eps < \kappa /|F| < 4q\eps$.

Since the action is free,
by the Ornstein--Weiss tower theorem (as formulated in Theorem~4.46 of \cite{KerLi16})
there exists a finite collection $\{ (M_i ,T_i )\}_{i\in I}$
of measurable towers such that the sets $T_i M_i$ for $i\in I$ are pairwise disjoint,
$\mu (\bigsqcup_{i\in I} T_i M_i ) \geq 1 - \eps /(2q)$, and $T_i$ is $(K',\delta' )$-invariant for
every $i$. By regularity we can find closed sets
$C_i \subseteq M_i$ with $\mu (M_i \setminus C_i )$ small enough to ensure
that $\mu (\bigsqcup_{i\in I} T_i C_i ) \geq 1 - \eps /q$. Then by compactness we can find
open sets $V_i \supseteq C_i$ such that the sets $T_i V_i$ for $i\in I$ are pairwise disjoint.

Let $i\in I$.
By our choice of $K'$ and $\delta'$ we can find pairwise disjoint sets 
$S_{i,0} , \dots , S_{i,m} \subseteq T_i$ all having the same cardinality $\kappa$ 
satisfying $2q\eps < \kappa /|T_i| < 4q\eps$.
Set $T_i' = S_{i,0} \sqcup\cdots\sqcup S_{i,m}$. Then 
\begin{align}\label{E-n}
|T_i' | = (m+1)\kappa < 4(m+1)q\eps |T_i | = \frac1n |T_i | .
\end{align}

Set $A = X\setminus \bigsqcup_{i\in I} T_i V_i$ and $B = \bigsqcup_{i\in I} S_{i,0} V_i$. 
Then for every $\nu\in M_G (X)$ we have, since $\nu$ is convex combination of the measures 
$\mu_1 , \dots , \mu _q$,
\begin{align*}
\nu (A) 
\leq \max_{k=1,\dots ,q} \mu_k (A) 
\leq q\mu (A)
\leq \eps ,
\end{align*}
from which we get $\nu (\bigsqcup_{i\in I} T_i V_i ) \geq 1 - \eps \geq 1/2$ and hence
\begin{align*}
\nu (B) 
\geq \sum_{i\in I} \frac{|S_{i,0} |}{|T_i|}\nu (T_i V_i ) 
> 2q\eps \nu \bigg( \bigsqcup_{i\in I} T_i V_i \bigg) 
\geq \frac{1}{4n(m+1)} 
\geq \eps .
\end{align*}
Since $A$ is closed and $B$ is open we thus have
$A\prec_m B$ by our $m$-comparison hypothesis. 
We can therefore find a finite collection $\cU$
of open subsets of $X$ which cover $A$, an $s_U \in G$ for each $U\in\cU$, and a partition 
$\cU = \cU_0 \sqcup\cdots\sqcup \cU_m$ such that for each $i=0,\dots ,m$ the images 
$s_U U$ for $U\in\cU_i$ are pairwise disjoint subsets of $B$.

For each $i\in I$ and $j=0,\dots ,m$ choose a bijection $\varphi_{i,j} : S_{i,0} \to S_{i,j}$.
For $U\in\cU$, $i\in I$, and $t\in S_{i,0}$ write $W_{U,i,t}$ for the open set
$U\cap s_U^{-1} tV_i$. For a fixed $U$, the sets $W_{U,i,t}$ for $i\in I$ and 
$t\in S_{i,0}$ partition $U$.
Moreover, writing $j_U$ for the $j$ such that $U\in\cU_j$,
the sets $\varphi_{i,j_U} (t)t^{-1} s_U W_{U,i,t}$
over all $U\in\cU$, $i\in I$, and $t\in S_{i,0}$ are pairwise disjoint and contained 
in $\bigsqcup_{i\in I} T_i' V_i$. This shows that 
\begin{align*}
A \prec \bigsqcup_{i\in I} T_i' V_i .
\end{align*}
Combined with (\ref{E-n}), this verifies almost finiteness.
\end{proof}

Combining Theorems~\ref{T-tower} and \ref{T-comparison af} yields:

\begin{theorem}\label{T-comparison af tower}
Suppose that $G$ is amenable.
Let $G\curvearrowright X$ be a free minimal action on a compact metrizable space
such that $E_G (X)$ is finite. Consider the following conditions:
\begin{enumerate}
\item $\towdim (X,G) < \infty$ and $dim (X) < \infty$,

\item $\ftowdim (X,G) < \infty$,

\item the action is almost finite,

\item the action has comparison.
\end{enumerate}
Then (i)$\Leftrightarrow$(ii)$\Rightarrow$(iii)$\Leftrightarrow$(iv).
\end{theorem}

The implication (ii)$\Rightarrow$(iii) in Theorem~\ref{T-comparison af tower} 
cannot be reversed, as the following examples show.
The obstruction in both cases is infinite-dimensionality,
whether in the space (Example~\ref{E-product}) or in the group (Example~\ref{E-asdim}).

\begin{example}\label{E-product}
Let $\{ \theta_k \}$ be a sequence of rationally independent numbers in $[0,1)$. 
Consider the product action $\Zb\stackrel{\alpha}{\curvearrowright} \prod_{k=0}^\infty X_k$ 
whose zeroeth factor is the odometer action $\Zb\curvearrowright \{ 0,1 \}^\Nb$
and whose $k$th factor for $k\geq 1$ is the action $(n,z) \mapsto e^{2\pi i n\theta_k } z$ on $\Tb$. 
This action is free. It is uniquely ergodic since each factor is uniquely ergodic
(as is well known) and the factors are mutually disjoint (because no two of them, when viewed
as measure-preserving actions with respect to the unique invariant Borel probability measure on each,
have a common eigenvalue except for $1$). It is minimal since the unique invariant
Borel probability measure on $\prod_{k=0}^\infty X_k$, i.e., the product of the 
unique invariant Borel probability measures on the factors, has full support.
It is also almost finite.
To see this, first note that the odometer action $\Zb\curvearrowright \{ 0,1 \}^\Nb$ is almost finite
since for every $n\in\Nb$ the clopen set $\{ 0 \}^{\{1,\dots , n\}} \times \{0,1 \}^{\{ n+1,n+2,\dots \}}$
is the base of a tower with shape $\{ 0,1,\dots , 2^n -1 \}$ whose levels partition $\{ 0,1 \}^\Nb$.
Now for $m\geq 1$ we can view $\Zb\curvearrowright \prod_{k=0}^m X_k$ 
as an extension of $\Zb\curvearrowright \prod_{k=0}^{m-1} X_k$ 
via the natural projection map, and so it follows by Theorem~\ref{T-af extension finite} 
and induction that 
$\Zb\curvearrowright \prod_{k=0}^m X_k$ is almost finite for every $m\geq 0$.
One can alternatively derive this conclusion by combining the fact that the odometer action
has tower dimension 1 (Example~\ref{E-Z Cantor}) with
Proposition~\ref{P-towdim} (tower dimension is nonincreasing under taking extensions)
and (i)$\Rightarrow$(iii) of Theorem~\ref{T-comparison af tower}.
It follows finally by Proposition~\ref{P-af inverse} that $\alpha$, 
being the inverse limit of the actions $\Zb\curvearrowright \prod_{k=0}^m X_k$, is almost finite.
This example shows that, for free minimal actions of $\Zb$, almost finiteness does not imply
finite tower dimension, since the latter implies that the space has finite covering dimension,
which is not the case here.
\end{example}

\begin{example}\label{E-asdim}
By Proposition~\ref{P-asymptotic dimension}, a necessary condition for a free
action $G\curvearrowright X$ to have finite tower dimension is that the group $G$ have
finite asymptotic dimension, which fails for many amenable groups, such as the Grigorchuk group.
Since every countably infinite amenable group admits almost finite free minimal actions
by Theorem~\ref{T-tiling}, this gives many examples of almost finite free minimal actions
which fail to have finite tower dimension.
\end{example}

\section{Disjointness in tower closures and almost finiteness in dimension zero}\label{S-zero dim}

In \cite{Mat12} Matui introduced a notion of almost finiteness for second countable
{\'e}tale groupoids with compact zero-dimensional unit spaces. 
We show in Theorem~\ref{T-zero dim} that when the groupoid arises from a free action 
$G\curvearrowright X$ on a zero-dimensional compact metrizable space, 
our notion of almost finiteness coincides with Matui's, justifying our use of the terminology.
What we in fact prove is that the action is almost finite (in the sense of Definition~\ref{D-af})
if and only if for every finite set $K\subseteq G$ 
and $\delta > 0$ there is a clopen castle (Definition~\ref{D-castle 2}) whose
shapes are $(K,\delta )$-invariant and whose levels partition $X$ (a clopen castle 
whose levels partition $X$ will be called a {\it clopen tower decomposition} of $X$). 
That this characterization is equivalent to Matui's almost finiteness
is recorded as Lemma~5.3 in \cite{Suz17}. 

The following lemma will be useful in establishing not only Theorem~\ref{T-zero dim}
but also Theorem~\ref{T-af Z-stable}.

\begin{lemma}\label{L-tower closures}
In Definition~\ref{D-af}
we may equivalently require each tower $(V_i ,S_i )$ to have the additional property 
that the sets $s\overline{V_i}$ for $s\in S_i$ are pairwise disjoint.
\end{lemma}

\begin{proof}
Let $G\curvearrowright X$ be a free action which is almost finite.
If $G$ is finite, then by taking
$n>|G|$, $K=G$, and $\delta < |G|^{-1}$ in Definition~\ref{D-af}
we are guaranteed the existence of an open castle $\{ (V_i , S_i ) \}_{i\in I}$ 
such that each shape is equal to $G$,
every level has diameter smaller than $\delta$, and
$\bigsqcup_{i\in I} S_i V_i = X$. It follows that every $V_i$ is clopen and so we obtain the
assertion of the lemma. We may thus assume that $G$ is infinite.

Let $K$ be a finite subset of $G$, $n\in\Nb$, and $\delta > 0$.
Since $G$ is infinite, there exists a finite set $K' \subset G$ with $K\subseteq K'$ 
and a $\delta' > 0$ with $\delta' \leq \delta$ such that every $(K' ,\delta' )$-invariant
nonempty finite subset of $G$ has cardinality greater than $2n$.
By almost finiteness there exist
\begin{enumerate}
\item an open castle $\{ (V_i ,S_i ) \}_{i\in I}$ whose shapes are $(K',\delta' )$-invariant
and whose levels have diameter less than $\delta$, 

\item sets $S_i' \subseteq S_i$ with $|S_i' | < |S_i |/(2n)$ such that
\[
X\setminus \bigsqcup_{i\in I} S_i V_i \prec \bigsqcup_{i\in I} S_i' V_i .
\]
\end{enumerate}
For each $i$ the set $S_i$ has cardinality greater than $2n$ by our choice of $K'$ and $\delta'$,
and so by setting $S_i'' = S_i' \cup \{ s \}$ for some arbitrarily chosen $s\in S_i \setminus S_i'$
we will have $|S_i'' | < |S_i |/n$.
Given a $\mu\in M_G (X)$, from (ii) we have
$\mu (X\setminus \bigsqcup_{i\in I} S_i V_i ) \leq \mu (\bigsqcup_{i\in I} S_i' V_i ) $,
which in particular implies that $\mu (V_i ) > 0$ for at least one $i\in I$, and hence that
\begin{align*}
\mu \bigg(X\setminus \bigsqcup_{i\in I} S_i V_i \bigg)
< \mu \bigg(\bigsqcup_{i\in I} S_i'' V_i \bigg) .
\end{align*}
It follows by Lemma~\ref{L-portmanteau} there is an $\eta > 0$ such that the sets
\begin{align*}
B &= \bigg\{ x\in X : d\bigg(x,X\setminus \bigsqcup_{i\in I} S_i'' V_i\bigg) > \eta \bigg\} , \\
A &= \bigg\{ x\in X : d\bigg(x,X\setminus \bigsqcup_{i\in I} S_i V_i\bigg) \leq \eta \bigg\}
\end{align*}
satisfy $\mu (A) < \mu (B)$ for all $\mu\in M_G (X)$.
By uniform continuity we can then find an $\eta' > 0$ such that 
the open sets 
\[
U_i = \{ x\in X : d(x,X\setminus V_i ) > \eta' \}
\]
for $i\in I$ satisfy $X\setminus \bigsqcup_{i\in I} S_i U_i \subseteq A$
and $B \subseteq \bigsqcup_{i\in I} S_i'' U_i$. 
Then for every $\mu\in M_G (X)$ we have
\begin{align*}
\mu \bigg(X\setminus \bigsqcup_{i\in I} S_i U_i \bigg)
\leq \mu (A) 
< \mu (B) 
\leq \mu \bigg(\bigsqcup_{i\in I} S_i'' U_i \bigg) .
\end{align*}
Since the action is almost finite, it has comparison by Theorem~\ref{T-comparison af},
and so we deduce that
\begin{align*}
X\setminus \bigsqcup_{i\in I} S_i U_i 
\prec \bigsqcup_{i\in I} S_i'' U_i
\end{align*}
Therefore the open castle $\{ (U_i ,S_i ) \}_{i\in I}$ and the sets $S_i'' \subseteq S_i$
witness the definition of almost finiteness with respect to $n$, $K$, and $\delta$,
and for each $i\in I$ the inclusion $\overline{U_i} \subseteq V_i$ 
implies that the sets $s\overline{U_i}$ for $s\in S_i$
are pairwise disjoint, as desired.
\end{proof}

\begin{theorem}\label{T-zero dim}
A free action $G\curvearrowright X$ on a zero-dimensional compact metric space
is almost finite if and only if
for every finite set $K\subseteq G$ and $\delta > 0$ there is a clopen castle
whose shapes are $(K,\delta )$-invariant and whose levels partition $X$.
\end{theorem}

\begin{proof}
The only issue in establishing the backward implication is arranging for the small diameter condition
in the definition of almost finiteness, and this can be done by observing that
for every $\eps > 0$ and clopen tower $(V,S)$ we can use uniform continuity to find a clopen partition
$\{ V_i \}_{i\in I}$ of $V$ such that the levels of the clopen castle $\{ (V_i ,S) \}_{i\in I}$,
which partition $SV$, all have diameter less than $\eps$.

For the forward implication, suppose that the action is almost finite.
Let $K$ be a finite subset of $G$ and $\delta > 0$. Take an $n\in\Nb$ such that
$2/n \leq \delta /2$.
By Lemma~\ref{L-tower closures} there are 
\begin{enumerate}
\item an open castle $\{ (V_i ,S_i ) \}_{i\in I}$ with $(K,\delta /2)$-invariant shapes
such that for each $i$ the sets $s\overline{V_i}$ for $s\in S_i$ are pairwise disjoint,

\item sets $S_i' \subseteq S_i$ such that $|S_i' | < |S_i |/n$ and 
the set $D:=X\setminus \bigsqcup_{i\in I} S_i V_i$ satisfies
\[
D \prec \bigsqcup_{i\in I} S_i' V_i .
\]
\end{enumerate}
Since the sets $s\overline{V_i}$ for $i\in I$ and $s\in S_i$ are closed,
by uniform continuity we can find, for each $i$, an open set $V_i' \supseteq \overline{V_i}$
such that the sets $sV_i'$ for $i\in I$ and $s\in S_i$ are pairwise disjoint.
Then, using compactness and zero-dimensionality, for each $i\in I$ we can cover
$\overline{V_i}$ with finitely many clopen subsets of $V_i'$, and so by replacing
$V_i$ with the union of these clopen sets we may assume that each of the sets $V_i$ is clopen.
Note in particular that the set $D = X\setminus \bigsqcup_{i\in I} S_i V_i$,
which is now clopen, still satisfies $D \prec \bigsqcup_{i\in I} S_i' V_i$,
since each new $V_i$ contains the original one. 
By Proposition~\ref{P-clopen subequivalence}
we can then find a clopen partition $\cU$ 
of $D$ and elements $t_U \in G$ for $U\in\cU$ such that the images
$t_U U$ for $U\in\cU$ are pairwise disjoint subsets of $\bigsqcup_{i\in I} S_i' V_i$.
We may assume,
by splitting each tower $(V_i ,S_i )$ into finitely many towers having the same shape $S_i$
and with bases forming a suitable clopen partition of $V_i$, 
that for every $U\in\cU$, $i\in I$, and $s\in S_i'$
such that $sV_i \cap t_U U \neq \emptyset$ we in fact have $sV_i \subseteq t_U U$.
By replacing $\cU$ with the clopen refinement consisting of the sets 
of the form $t_U^{-1} sV_i$ where $sV_i$ is a tower level which is contained in $t_U U$ for some $U\in\cU$,
we may now also assume that for every $U\in\cU$ there are an $i_U \in I$ and an $s_U \in S_{i_U}'$
such that $t_U U = s_U V_{i_U}$. 

Let $i\in I$. Set $S_i'' = \{ t_U^{-1} s_U : U\in\cU \text{ and } s_U \in S_i' \}$.
Note that the map $U\mapsto t_U^{-1} s_U$ from $\{ U\in\cU : i_U = i \}$ to $S_i'$ is injective,
for if $t_U^{-1} s_U = t_{U'}^{-1} s_{U'}$ for $U$ and $U'$ in the domain then
$U = t_U^{-1} s_U V_i = t_{U'}^{-1} s_{U'} V_i = U'$.
Thus $|S_i'' | \leq |S_i' |$. Define $\tilde{S}_i = S_i \sqcup S_i''$. Then for every $t\in K$ we have
\begin{align*}
|t\tilde{S}_i \Delta \tilde{S}_i |
&\leq |tS_i \Delta S_i | + |tS_i'' | + |S_i'' | \\
&< \frac{\delta}{2} |S_i | + 2|S_i' | \\
&\leq \bigg(\frac{\delta}{2} + \frac{2}{n} \bigg) |S_i | \\
&\leq \delta |\tilde{S}_i | ,
\end{align*}
showing that $\tilde{S}_i$ is $(K,\delta )$-invariant. Therefore 
$\{ (V_i ,\tilde{S}_i ) \}_{i\in I}$ is clopen tower
decomposition of $X$ with $(K,\delta )$-invariant shapes, as desired.
\end{proof}

\begin{remark}\label{R-nilpotent}
Matui showed in \cite{Mat12} that almost finiteness for a second countable
{\'e}tale groupoid $G$ with compact zero-dimensional unit space
has several implications for the homology groups $H^n (G)$ and their relation to
both the topological full group $\llbracket G\rrbracket$ 
and the $K$-theory of the reduced groupoid C$^*$-algebra of $G$.
In particular, if the groupoid is principal and almost finite then there is
a canonical isomorphism $H^1 (G) \cong \llbracket G\rrbracket /N$ 
where $N$ is the subgroup generated by the elements of finite order
(see Section~7 of \cite{Mat12}). As Matui observes in Lemma~6.3 of \cite{Mat12},
the groupoid associated to a free action of $\Zb^m$ on a zero-dimensional
compact metrizable space is almost finite. By Theorem~\ref{T-zero dim},
Example~\ref{E-nilpotent}, and Theorem~\ref{T-comparison af tower},
we see that this is also the case for every free minimal action $G\curvearrowright X$
of a finitely generated nilpotent group on a zero-dimensional compact metrizable space
with $E_G (X)$ finite.
\end{remark}

\section{Almost finiteness and extensions}\label{S-af extn}

As noted in Remark~\ref{R-af extn}, almost finiteness does not pass to extensions in general.
We will show however in Theorem~\ref{T-af extension finite} that an extension $G\curvearrowright Y$
of an almost finite free action $G\curvearrowright X$ is again almost finite
whenever $E_G (Y)$ and $\dim (Y)$ are both finite.
To this end we will employ the following notions of coarse almost finiteness and $m$-almost finiteness.

\begin{definition}\label{D-coarsely af}
We say that a free action $G\curvearrowright X$ on a compact metric space
is {\it coarsely almost finite} if 
for every $n\in\Nb$, finite set $K\subseteq G$, and $\delta > 0$ there are 
\begin{enumerate}
\item a collection $\{ (V_i ,S_i ) \}_{i\in I}$ of open towers with $(K,\delta )$-invariant
shapes such that $\{ S_i \overline{V_i} \}_{i\in I}$ is a castle,

\item sets $S_i' \subseteq S_i$ such that $|S_i' | < |S_i |/n$ and
\[
X\setminus \bigsqcup_{i\in I} S_i V_i \prec \bigsqcup_{i\in I} S_i' V_i .
\]
\end{enumerate}
\end{definition}

An almost finite free action is coarsely almost finite by Lemma~\ref{L-tower closures}.

\begin{definition}\label{D-m-af}
Let $m\in\Nb$.
We say that a free action $G\curvearrowright X$ on a compact metric space
is {\it $m$-almost finite} if for every $n\in\Nb$, finite set $K\subseteq G$, and $\delta > 0$
there are 
\begin{enumerate}
\item a collection $\{ (V_i ,S_i ) \}_{i\in I}$ of open towers with $(K,\delta )$-invariant
shapes such that $\diam (sV_i ) < \delta$ for every $i\in I$ and $s\in S_i$ 
and the family $\{ S_i V_i \}_{i\in I}$ has chromatic number at most $m+1$,

\item sets $S_i' \subseteq S_i$ such that $|S_i' | < |S_i |/n$ and
\[
X\setminus \bigsqcup_{i\in I} S_i V_i \prec \bigsqcup_{i\in I} S_i' V_i .
\]
\end{enumerate}
\end{definition}

The following is easily verified by taking the inverse images under the extension $Y\to X$
of all of the sets at play in the definition of coarse almost finiteness. 

\begin{proposition}\label{P-coarsely af extn}
If $Y\to X$ is an extension of free actions of $G$ and $G\curvearrowright X$ is coarsely
almost finite, then $G\curvearrowright Y$ is coarsely almost finite.
\end{proposition}

\begin{lemma}\label{L-fd comparison}
Suppose that $X$ has covering dimension $d<\infty$ and let $G\curvearrowright X$ be a free action
which is coarsely almost finite. Then the action is $d$-almost finite.
\end{lemma}

\begin{proof}
Let $n\in\Nb$, and let $K$ be a finite subset of $G$ and $\delta > 0$.
By coarse almost finiteness there are 
\begin{enumerate}
\item a collection $\{ (V_i ,S_i ) \}_{i\in I}$ of open towers with $(K,\delta )$-invariant shapes
such that $\{ S_i \overline{V_i} \}_{i\in I}$ is a castle, and

\item sets $S_i' \subseteq S_i$ such that $|S_i' | < |S_i |/n$ and
\[
X\setminus \bigsqcup_{i\in I} S_i V_i \prec \bigsqcup_{i\in I} S_i' V_i .
\]
\end{enumerate}
Since the towers $S_i \overline{V_i}$ for $i\in I$ are pairwise disjoint, 
for each $i$ we can find an open set $U_i \supseteq V_i$ so that the towers
$S_i U_i$ for $i\in I$ are still pairwise disjoint. Since $X$ has covering dimension
$d$, for each $i$ we can find a collection $\{ V_{i,1} , \dots ,V_{i,k_i} \}$ of open
subsets of $U_i$ which covers $\overline{V_i}$, satisfies $\diam (sV_{i,j} ) < \delta$
for every $j=1,\dots , k_i$, and has chromatic number at most $d+1$.
The collection of towers $\{ (V_{i,j} , S_i ) : i\in I,\, 1\leq j\leq k_i \}$
then fulfills the requirements in the definition of $d$-almost finiteness.
\end{proof}

The proof of the following is essentially the same as for (i)$\Rightarrow$(ii)
of Theorem~\ref{T-comparison af}, which is the case $m=0$. 
We leave the details to the reader.

\begin{lemma}\label{L-af comparison}
Let $G\curvearrowright X$ be a free action which is $m$-almost finite. Then the action 
has $m$-comparison.
\end{lemma}

\begin{theorem}\label{T-af extension finite}
Let $G\stackrel{\alpha}{\curvearrowright} X$ be an almost finite free action and let 
$G\stackrel{\beta}{\curvearrowright} Y$ be an extension of $\alpha$ such that $E_G (Y)$ is finite
and $\dim (Y) < \infty$.
Then $\beta$ is almost finite.
\end{theorem}

\begin{proof}
Since $\alpha$ is almost finite it is coarsely almost finite, 
and so by Proposition~\ref{P-coarsely af extn} the action $\beta$ is 
coarsely almost finite. Consequently $\beta$ has $m$-comparison by 
Lemmas~\ref{L-af comparison} and \ref{L-fd comparison}.
We then conclude by Theorem~\ref{T-comparison af} that $\beta$ is almost finite.
\end{proof}

\section{Almost finiteness and $\cZ$-stability}\label{S-af Z-stable}

We show here in Theorem~\ref{T-af Z-stable} that, assuming $G$ is infinite,
the reduced crossed product 
$C(X)\rtimes_\lambda G$ of an almost finite free minimal action $G\curvearrowright X$
on a compact metrizable space is $\cZ$-stable. The argument uses tiling technology
as in the proof of Theorem~5.3 of \cite{ConJacKerMarSewTuc17}.
Note that since almost finiteness implies that
$G$ is amenable, the reduced and full crossed products coincide in this case,
although we will not need this fact. 

Recall that {\it c.p.c.}\ stands for ``completely positive contractive'', and that
a map $\varphi : A\to B$ between C$^*$-algebras is {\it order-zero} if
$\varphi (a_1 )\varphi (a_2 ) = 0$ for all $a_1 , a_2 \in A$ satisfying $a_1 a_2 = 0$.
We write $\precsim$ for the relation of Cuntz subequivalence.

In order to verify $\cZ$-stability we will use the following result 
of Hirshberg and Orovitz (Theorem~4.1 of \cite{HirOro13}).

\begin{theorem}\label{T-Z-stable}
Let $A$ be a simple separable unital nuclear C$^*$-algebra not isomorphic to $\Cb$. 
Suppose that for every
$n\in\Nb$, finite set $\Omega\subseteq A$, $\eps > 0$, and nonzero positive element $a\in A$ 
there exists an order-zero c.p.c.\ map
$\varphi : M_n \to A$ such that
\begin{enumerate}
\item $1-\varphi (1) \precsim a$,

\item $\| [a,\varphi (b)] \| < \eps$ for all $a\in\Omega$ and norm-one $b\in M_n$.
\end{enumerate}
Then $A$ is $\cZ$-stable.
\end{theorem}

The following is the Ornstein--Weiss quasitiling theorem \cite{OrnWei87}. See
Theorem~4.36 of \cite{KerLi16} for this precise formulation. For $0\leq \eta\leq 1$,
we say that a collection $\{ A_i \}$ of subsets of a finite set $E$ is {\it $\eta$-disjoint}
if there exist sets $A_i' \subseteq A_i$ with $|A_i' | \geq (1-\eta )|A_i |$ such that
the collection $\{ A_i' \}$ is disjoint, and that it {\it $\eta$-covers} $E$ if
$|\bigcup_i A_i | \geq \eta |E|$.

\begin{theorem}\label{T-qt}
Let $0<\beta <\frac12$ and let $n\in\Nb$ be such that
$(1-\beta /2)^n < \beta$. Then
whenever $e\in T_1 \subseteq T_2 \subseteq \dots \subseteq T_n$
are finite subsets of a group $G$ such that $|\partial_{T_{i-1}} T_i | \leq (\beta /8)|T_i |$ for
$i=2, \dots , n$, for every $(T_n ,\beta /4)$-invariant nonempty finite set $E\subseteq G$
there exist $C_1 , \dots , C_n \subseteq G$ such that
\begin{enumerate}
\item $\bigcup_{i=1}^n T_i C_i \subseteq E$, and

\item the collection of right translates $\bigcup_{i=1}^n \{ T_i c : c\in C_i \}$ is
$\beta$-disjoint and $(1-\beta )$-covers $E$.
\end{enumerate}
\end{theorem}

\begin{lemma}\label{L-partial}
Let $G\curvearrowright X$ be an action on a compact metrizable space.
Let $A$ be a closed subset of $X$ and $B$ an open subset of $X$ such that $A\prec B$.
Let $f,g:X\to [0,1]$ be continuous functions such that $f=0$ on $X\setminus A$
and $g=1$ on $B$. Then there is a $v\in C(X)\rtimes_\lambda G$
such that $v^* gv = f$. 
\end{lemma}

\begin{proof}
As $A\prec B$ there exist open sets $U_1 , \dots , U_n \subseteq X$
such that $A\subseteq \bigcup_{i=1}^n U_i$ and an $s_i \in G$ for each $i=1,\dots ,n$ such that the images 
$s_i U_i$ for $i=1,\dots ,n$ are pairwise disjoint subsets of $B$.
In the same way that one constructs a partition of unity subordinate to a given open cover,
we can produce, for each $i=1,\dots ,n$, a continuous function $h_i : X\to [0,1]$
with $h_i = 0$ on $X\setminus U_i$ so that $0\leq \sum_{i=1}^n h_i \leq 1$ and
$\sum_{i=1}^n h_i = 1$ on $A$.
Set $v = \sum_{i=1}^n u_{s_i} (fh_i )^{1/2}$. 

Denote by $\alpha$ the induced action of $G$ on $C(X)$,
that is, $\alpha_s (f)(x) = f(s^{-1} x)$ for all $s\in G$, $f\in C(X)$, and $x\in X$.
Since $\alpha_{s_i} (h_i^{1/2} )\alpha_{s_j} (h_j^{1/2} )= 0$ for $i\neq j$
and $g$ dominates $\alpha_{s_i} (h_i^{1/2} )$ for every $i$, 
we have
\begin{align*}
v^* gv 
&= \bigg( \sum_{i=1}^n (h_i f)^{1/2} u_{s_i}^* \bigg) g\bigg( \sum_{i=1}^n u_{s_i} (fh_i )^{1/2} \bigg) \\
&= \bigg( \sum_{i=1}^n u_{s_i}^* \alpha_{s_i} (f^{1/2})\alpha_{s_i} (h_i^{1/2} ) \bigg) 
g\bigg( \sum_{i=1}^n \alpha_{s_i} (h_i^{1/2} )\alpha_{s_i} (f^{1/2}) u_{s_i} \bigg) \\
&= \sum_{i=1}^n u_{s_i}^* \alpha_{s_i} (fh_i) u_{s_i} \\
&= \sum_{i=1}^n u_{s_i}^* \alpha_{s_i} (fh_i) u_{s_i} 
= \sum_{i=1}^n fh_i 
= f ,
\end{align*}
as desired.
\end{proof}

\begin{theorem}\label{T-af Z-stable}
Suppose that $G$ is infinite.
Let $G\curvearrowright X$ be a free minimal action which is almost finite.
Then $C(X)\rtimes_\lambda G$ is $\cZ$-stable.
\end{theorem}

\begin{proof}
As before we denote the induced action of $G$ on $C(X)$ by $\alpha$,
that is, $\alpha_s (f)(x) = f(s^{-1} x)$ for all $s\in G$, $f\in C(X)$, and $x\in X$.

Let $n\in\Nb$. Let $\Upsilon$ be a finite subset of the unit ball of $C(X)$, 
$F$ a symmetric finite subset of $G$ containing $e$, 
and $\eps > 0$. Let $a$ be a nonzero positive element of $C(X)\rtimes G$.
We will show the existence of a map $\varphi : M_n \to C(X)\rtimes_\lambda G$ 
as in Theorem~\ref{T-Z-stable}
where the finite set $\Omega$ there is taken to be $\Upsilon\cup \{ u_s : s\in F \}$.
Since $C(X)\rtimes_\lambda G$ is generated as a C$^*$-algebra by the unit ball of $C(X)$ 
and the unitaries $u_s$ for $s\in G$, we will thereafter be able to conclude by Theorem~\ref{T-Z-stable} 
that $C(X)\rtimes_\lambda G$ is $\cZ$-stable.

By Lemma~7.9 in \cite{Phi14} we may assume that $a\in C(X)$. Then we can find an
$x_0 \in X$ and a $\theta > 0$ such that $a$ is strictly positive on the closed ball
of radius $3\theta$ centred at $x_0$. We may therefore assume that $a$
is a $[0,1]$-valued function which takes value $1$ on all points within distance $2\theta$ from $x_0$
and value $0$ at all points at distance at least $3\theta$ from $x_0$.
Write $O$ for the open ball of radius $\theta$ centred at $x_0$.
Minimality implies that the sets $sO$ for $s\in G$ cover $X$,
and so by compactness there is a finite set $D\subseteq G$ such that $D^{-1} O = X$.

Let $0 < \kappa < 1$, to be determined. Choose an integer $Q > n^2 /\eps$.
Take a $\beta > 0$ 
which is small enough so that if $T$ is a nonempty finite subset of $G$
which is sufficiently invariant under left translation by $F^Q$ then
for every set $T' \subseteq T$ with $|T' | \geq (1-n\beta )|T|$ one has
\begin{align*}
\bigg|\bigcap_{s\in F^Q} sT' \bigg| \geq (1-\kappa )|T| .
\end{align*}

Choose an $L\in\Nb$ large enough so that $(1-\beta /2)^L < \beta$.
Since $G$ is amenable by the almost finiteness of the action,
there exist finite subsets $e\in T_1 \subseteq T_2 \subseteq\dots\subseteq T_L$ of $G$ 
such that $|\partial_{T_{l-1}} T_l | \leq (\beta /8)|T_l |$ for $l=2,\dots , L$. 
By the previous paragraph, we may also assume that for each $l$ the set $T_l$ 
is sufficiently invariant under left translation by $F^Q$ so that 
\begin{align}\label{E-kappa}
\bigg|\bigcap_{s\in F^Q} sT \bigg| \geq (1-\kappa )|T_l | 
\end{align}
for every $T \subseteq T_l$ satisfying $|T| \geq (1-n\beta )|T_l |$. 

By the uniform continuity of functions in $\Upsilon\cup\Upsilon^2$
and the uniform continuity of the transformations $x\mapsto tx$ of $X$ for $t\in T_L$,
there is an $\eta > 0$ such that if $d(x,y) < \eta$ then 
$|f(tx) - f(ty)| < \eps /(4n^2 )$ for all $f\in\Upsilon\cup\Upsilon^2$ and $t\in T_L$.
Let $\cU = \{ U_1 , \dots , U_M \}$ be an open cover of $X$ whose members all have diameter
less that $\eta$. Let $\eta' > 0$ be a Lebesgue number for $\cU$ which is no larger than $\theta$.

Let $E$ be a finite subset of $G$ containing $T_K$ and let $\delta > 0$ be such that $\delta \leq \beta /4$.
Since $G$ is infinite, we may enlarge $E$ and shrink $\delta$ as necessary so as to guarantee
that the cardinality of every nonempty $(E,\delta )$-invariant finite set $S\subseteq G$ is 
large enough to satisfy
\begin{align}\label{E-large card}
\bigg(\sum_{l=1}^L |T_l |\bigg)Mn \leq \beta |S| .
\end{align}

Since the action is almost finite, by Lemma~\ref{L-tower closures} we can find
\begin{enumerate}
\item nonempty open sets $V_1 , \dots , V_K \subseteq X$
and nonempty $(E,\delta )$-invariant finite sets $S_1 , \dots , S_K \subseteq G$
such that the family $\{ (\overline{V_k} ,S_k ) \}_{k=1}^K$ is a castle 
with levels of diameter less than $\eta'$, and

\item sets $S_k' \subseteq S_k$ such that $|S_k' |/|S_k | < 1/(4|D|^2)$ and
\begin{align}\label{E-prec}
X\setminus \bigsqcup_{k=1}^K S_k V_k \prec \bigsqcup_{k=1}^K S_k' V_k .
\end{align}
\end{enumerate}

Let $k\in \{ 1,\dots , K\}$. Since $S_k$ is $(T_L ,\beta /4)$-invariant, 
by Theorem~\ref{T-qt} and our choice of the sets $T_1 , \dots , T_L$ we can find 
$C_{k,1} ,\dots , C_{k,L} \subseteq S_k$ such that the collection 
$\{ T_l c : l=1,\dots ,L, \, c\in C_{k,l} \}$ is $\beta$-disjoint and
$(1-\beta )$-covers $S_k$. By $\beta$-disjointness, 
for every $l=1,\dots ,L$ and $c\in C_{k,l}$ we can find a $T_{k,l,c} \subseteq T_l$ 
satisfying $|T_{k,l,c} | \geq (1-\beta )|T_l |$
so that the collection of sets $T_{k,l,c} c$ 
for $l=1,\dots ,L$ and $c\in C_{k,l}$ is disjoint.

Since $\eta'$ is a Lebesgue number for $\cU$ and the levels
of the tower $(V_k ,S_k )$ have diameter less than $\eta'$,
for each $l=1,\dots ,L$ there is a partition
\begin{align*}
C_{k,l} = C_{k,l,1} \sqcup C_{k,l,2} \sqcup\cdots\sqcup C_{k,l,M}
\end{align*}
such that $cV_k \subseteq U_m$ for all $m=1,\dots ,M$ and $c\in C_{k,l,m}$.
For each $l$ and $m$ choose pairwise disjoint subsets
$C_{k,l,m}^{(1)} , \dots , C_{k,l,m}^{(n)}$ of $C_{k,l,m}$ such that each has cardinality 
$\lfloor |C_{k,l,m} | /n \rfloor$.
For each $i=2,\dots , n$ choose a bijection 
\begin{align*}
\Lambda_{k,i} : \bigsqcup_{l,m} C_{k,l,m}^{(1)} \to \bigsqcup_{l,m} C_{k,l,m}^{(i)}
\end{align*}
which sends $C_{k,l,m}^{(1)}$ to $C_{k,l,m}^{(i)}$ for all $l,m$.
Also, define $\Lambda_{k,1}$ to be the identity map from $\bigsqcup_{l,m} C_{k,l,m}^{(1)}$ to itself,
and write $\Lambda_{k,i,j}$ for the composition $\Lambda_{k,i} \circ\Lambda_{k,j}^{-1}$.

Now consider for each $j=1,\dots ,n$ and $c\in C_{k,l,m}^{(j)}$ the set 
$T_{k,l,c}' := \bigcap_{i=1}^n T_{k,l,\Lambda_{k,i,j} (c)}$, which satisfies
\begin{align}\label{E-n beta}
|T_{k,l,c}' | \geq (1-n\beta )|T_l | .
\end{align}
since each $T_{k,l,\Lambda_{k,i} (c)}$ is a subset of $T_l$ with cardinality at least $(1-\beta )|T_l |$.
Set 
\[
B_{k,l,c,Q} = \bigcap_{s\in F^Q} sT_{k,l,c}'
\]
and for $q=0,\dots, Q-1$ put
\begin{align*}
B_{k,l,c,q} = F^{Q-q} B_{k,l,c,Q} \setminus F^{Q-q-1} B_{k,l,c,Q} .
\end{align*}
Then the sets $B_{k,l,c,0} , \dots , B_{k,l,c,Q}$ partition $F^Q B_{k,l,c,Q}$, 
which is a subset of $T_{k,l,c}'$.
For $s\in F$ it is clear that 
\begin{align}\label{E-B_Q}
sB_{k,l,c,Q} \subseteq B_{k,l,c,Q-1} \cup B_{k,l,c,Q} ,
\end{align}
while for $q=1,\dots , Q-1$ we have
\begin{align}\label{E-B_q}
sB_{k,l,c,q} \subseteq B_{k,l,c,q-1} \cup B_{k,l,c,q} \cup B_{k,l,c,q+1} ,
\end{align}
for if we are given a $t\in B_{k,l,c,q}$ then $st\in F^{Q-q+1} B_{k,l,c,Q}$, while if
$st\in F^{Q-q-2} B_{k,l,c,Q}$ then $t\in F^{Q-q-1} B_{k,l,c,Q}$ since $F$ is symmetric,
contradicting the membership of $t$ in $B_{k,l,c,q}$.

We view $C(X)\rtimes_\lambda G$ as being canonically included in the crossed product $B(X)\rtimes_\lambda G$
of the action induced by $G\curvearrowright X$ on the C$^*$-algebra $B(X)$ of bounded Borel functions on $X$. 
Since the sets $s\overline{V_k}$ for $k=1,\dots ,K$ and $s\in S_k$ are closed and pairwise disjoint, 
for each $k$ we can find an open set $U_k \supseteq \overline{V_k}$ such that
the sets $sU_k$ for $k=1,\dots ,K$ and $s\in S_k$ are pairwise disjoint.
We define a linear map $\psi : M_n \to B(X)\rtimes_\lambda G$ by declaring it on the standard
matrix units $\{ e_{ij} \}_{i,j=1}^n$ of $M_n$ to be given by
\begin{align*}
\psi (e_{ij} ) = \sum_{k=1}^K \sum_{l=1}^L \sum_{m=1}^M \sum_{c\in C_{k,l,m}^{(j)}} \sum_{t\in T_{k,l,c}'} 
u_{t\Lambda_{k,i,j} (c)c^{-1} t^{-1}} \unit_{tcU_k}
\end{align*}
and extending linearly. 

For each $k=1,\dots ,K$ choose a continuous function $h_k : X\to [0,1]$ such that $h_k = 1$ on $\overline{V_k}$
and $h_k = 0$ on $X\setminus U_k$.
Recalling that $\alpha$ denotes the induced action of $G$ on $C(X)$,
for all $k$, $l$, and $m$, all $1\leq i,j\leq n$, and all $c\in C_{k,l,m}^{(j)}$ we set
\begin{align*}
h_{k,l,c,i,j} = \sum_{q=1}^Q \sum_{t\in B_{k,l,c,q}} 
\frac{q}{Q} u_{t\Lambda_{k,i,j} (c)c^{-1} t^{-1}} \alpha_{tc} (h_k ) .
\end{align*}
Define a linear map $\varphi : M_n \to C(X)\rtimes_\lambda G$ by setting
\begin{align*}
\varphi (e_{ij} ) = \sum_{k=1}^K \sum_{l=1}^L \sum_{m=1}^M \sum_{c\in C_{k,l,m}^{(j)}} h_{k,l,c,i,j} 
\end{align*}
and extending linearly.
Note that if we put 
\begin{align*}
h = \sum_{k=1}^K \sum_{l=1}^L \sum_{m=1}^M \sum_{i=1}^n \sum_{c\in C_{k,l,m}^{(i)}} h_{k,l,c,i,i} .
\end{align*}
then $h$ is a continuous function taking values in $[0,1]$ which commutes with the 
image of $\psi$, and we have
\begin{align*}
\varphi (b) = h\psi (b) 
\end{align*}
for all $b\in M_n$, which shows that $\varphi$ is an order-zero c.p.c.\ map. 

We now verify condition (ii) in Theorem~\ref{T-Z-stable} for the elements of the set 
$\{ u_s : s\in F \}$. Let $1\leq i,j\leq n$. 
For $s\in F$ we have
\begin{align*}
u_s h_{k,l,c,i,j} u_s^{-1} - h_{k,l,c,i,j} 
&= \sum_{q=1}^Q \sum_{t\in B_{k,l,c,q}} 
\frac{q}{Q} u_{st\Lambda_{k,i,j} (c)c^{-1} (st)^{-1}} \alpha_{stc} (h_k ) \\
&\hspace*{20mm} \ - \sum_{q=1}^Q \sum_{t\in B_{k,l,c,q}} 
\frac{q}{Q} u_{t\Lambda_{k,i,j} (c)c^{-1} t^{-1}} \alpha_{tc} (h_k ) ,
\end{align*}
and so in view of (\ref{E-B_Q}) and (\ref{E-B_q}) we obtain
\begin{align*}
\| u_s h_{k,l,c,i,j} u_s^{-1} - h_{k,l,c,i,j} \| \leq \frac{1}{Q} < \frac{\eps}{n^2} .
\end{align*}
Since the element $a = u_s h_{k,l,c,i,j} u_s^{-1} - h_{k,l,c,i,j}$ satisfies 
$a^* a \leq \unit_{F^Q B_{k,l,c,Q} c U_k}$ and $aa^* \leq \unit_{F^Q B_{k,l,\Lambda (c),Q} c U_k}$
and the sets $F^Q B_{k,l,c,Q} c U_k$ are pairwise disjoint for all $k$, $l$, and $c$, 
this yields
\begin{align*}
\| u_s \varphi (e_{ij} ) u_s^{-1} - \varphi (e_{ij} ) \|
= \max_{k,l,c} \| u_s h_{k,l,c,i,j} u_s^{-1} - h_{k,l,c,i,j} \| < \frac{\eps}{n^2}
\end{align*}
and hence, for every norm-one $b = (b_{ij} ) \in M_n$,
\begin{align*}
\| [ u_s ,\varphi (b)] \| 
&= \| u_s \varphi (b) u_s^{-1} - \varphi (b) \| \\
&\leq \sum_{i,j=1}^n \| u_s \varphi (b_{ij} ) u_s^{-1} - \varphi (b_{ij} ) \| 
< n^2 \cdot \frac{\eps}{n^2} = \eps .
\end{align*}

Next we verify condition (ii) in Theorem~\ref{T-Z-stable} for the functions in $\Upsilon$.
Let $1\leq i,j\leq n$.
Let $f\in\Upsilon\cup\Upsilon^2$. Let $1\leq k\leq K$ and $1\leq l\leq L$.
Let $c\in C_{k,l,m}^{(j)}$.
Since the elements $t\Lambda_{k,i,j} (c)$ for $t\in T_{k,l,c}'$ are distinct, we have
\begin{align}\label{E-two}
h_{k,l,c,i,j}^* fh_{k,l,c,i,j}
&= \sum_{q=1}^Q \sum_{t\in B_{k,l,c,q}} 
\frac{q^2}{Q^2} \alpha_{tc\Lambda_{k,i,j} (c)^{-1} t^{-1} } (f) \alpha_{tc} (h_k^2 )
\end{align}
and similarly
\begin{align}\label{E-one}
fh_{k,c,i,j}^* h_{k,c,i,j}
&= \sum_{q=1}^Q \sum_{t\in B_{k,l,c,q}} 
\frac{q^2}{Q^2} f\alpha_{tc} (h_k^2 )
\end{align}
Now let $x\in V_k$. Since $\Lambda_{k,i,j} (c)x$ and $cx$ both belong to $U_m$ 
by our definition of $C_{k,l,m}$,
we have $d(\Lambda_{k,i,j} (c)x,cx) < \eta$. It follows 
that for every $t\in T_l$ we have
\begin{align*}
|f(t\Lambda_{k,i,j} (c)x) - f(tcx)| < \frac{\eps}{4n^2} ,
\end{align*}
in which case
\begin{align*}
\| \alpha_{tc\Lambda_{k,i,j} (c)^{-1} t^{-1} } (f) - f \|
&= \| \alpha_{c^{-1} t^{-1}} (\alpha_{tc\Lambda_{k,i,j} (c)^{-1} t^{-1}} (f) - f) \| \\
&= \sup_{x\in V_k} |f(t\Lambda_{k,i,j} (c)x) - f(tcx)| \\
&< \frac{\eps}{3n^2} . 
\end{align*}
Using (\ref{E-two}) and (\ref{E-one}) this gives us
\begin{align}\label{E-max}
\lefteqn{\| h_{k,l,c,i,j}^* fh_{k,l,c,i,j} - fh_{k,l,c,i,j}^* h_{k,l,c,i,j} \|}\hspace*{20mm} \\
\hspace*{20mm} &= \max_{q=1,\dots ,Q} \max_{t\in B_{k,l,c,q}} \frac{q^2}{Q^2} 
\| (\alpha_{tc\Lambda_{k,i,j} (c)^{-1} t^{-1} } (f) - f)\alpha_{tc} (h_k^2 ) \| \notag \\
&< \frac{\eps}{3n^2} .\notag 
\end{align}
Set $w = \varphi (e_{ij} )$ for brevity. Let $f\in\Upsilon$.
Since the functions $h_{k,l,c,i,j}$ for $1\leq k\leq K$, $1\leq l\leq L$, $1\leq m\leq M$,
and $c\in C_{k,l,m}^{(j)}$ have pairwise disjoint supports, we infer 
from (\ref{E-max}) that
\begin{align*}
\| w^* gw - gw^* w \| < \frac{\eps}{3n^2} 
\end{align*}
for $g$ equal to either $f$ or $f^2$.
It follows that 
\begin{align*}
\| w^* f^2 w - fw^* fw \|
\leq \| w^* f^2 w - f^2 w^* w \| + \| f(fw^* w - w^* fw) \|
< \frac{2\eps}{3n^2}
\end{align*}
and hence
\begin{align*}
\| fw - wf \|^2
&= \| (fw-wf)^* (fw - wf) \| \\
&= \| w^* f^2 w - fw^* fw + fw^* wf - w^* fwf \| \\
&\leq \| w^* f^2 w - fw^* fw \| + \| (fw^* w - w^* fw)f \| \\
&< \frac{2\eps}{3n^2} + \frac{\eps}{3n^2} = \frac{\eps}{n^2}.
\end{align*}
Therefore for every norm-one $b= (b_{ij} ) \in M_n$ we have
\begin{align*}
\| [f,\varphi (b)] \| 
\leq \sum_{i,j=1}^n \| [f,\varphi (b_{ij} )] \|
< n^2 \cdot \frac{\eps}{n^2} = \eps . 
\end{align*}

To complete the proof, let us now show that $1-\varphi (1) \precsim a$.
By enlarging $E$ and shrinking $\delta$ if necessary we may assume
that the sets $S_1 , \dots , S_K$ are sufficiently left invariant 
so that for every $k=1,\dots ,K$ there is an $R_k \subseteq S_k$
such that the set $\{ s\in R_k : Ds\subseteq S_k \}$ has cardinality
at least $|S_k|/2$. Let $1\leq k\leq K$. Let $R_k'$ be a maximal subset of $R_k$
with the property that the sets $Ds$ for $s\in R_k'$ 
are pairwise disjoint. Observe that if $s,t\in R_k$ satisfy $Ds \cap Dt \neq \emptyset$
then $s\in D^{-1} Dt$, which shows that $|R_k' |\geq |R_k |/|D^{-1} D| \geq |S_k |/(2|D|^2 )$.
Since $D^{-1} O = X$, for each $s\in R_k'$ there is a $t\in D$ such that $tsV_k$ intersects $O$,
which implies that the function $a$ takes the constant value $1$ on $tsV_k$ since the diameter
of the latter set is less than $\theta$ and $a$ takes value $1$ at all points within 
distance $\theta$ of the set $O$.
Therefore the set $S_k^\sharp$ of all $t\in S_k$ such that $a$ takes the constant value $1$ on
$tV_k$ has cardinality at least $|S_k|/(2|D|^2 )$. Set 
\begin{align*}
S_k'' &= \bigsqcup_{l=1}^L \bigsqcup_{m=1}^M \bigsqcup_{i=1}^n
\bigsqcup_{c\in C_{k,l,m}^{(i)}} B_{k,l,c,Q} c .
\end{align*}
Since  $|B_{k,l,c,Q} | \geq (1-\kappa )|T_l | \geq (1-\kappa )|T_{k,l,m} |$ 
by (\ref{E-kappa}) and (\ref{E-n beta}),
using (\ref{E-large card}) 
we obtain
\begin{align*}
|S_k'' |
&\geq \sum_{l=1}^L \sum_{m=1}^M \sum_{i=1}^n \sum_{c\in C_{k,l,m}^{(i)}} |B_{k,l,c,Q} | \\
&\geq (1-\kappa ) \sum_{l=1}^L \sum_{m=1}^M \sum_{i=1}^n |T_{k,l,m} ||C_{k,l,m}^{(i)} |  \notag \\
&\geq (1-\kappa ) \sum_{l=1}^L \sum_{m=1}^M |T_{k,l,m} | ( |C_{k,l,m} | - n )  \notag \\
&\geq (1-\kappa ) \bigg(\bigg|\bigsqcup_{l=1}^L T_{k,l,m} C_{k,l} \bigg| - Mn\sum_{l=1}^L |T_l | \bigg)  \notag \\
&\geq (1-\kappa ) (1-2\beta )|S_k | , \notag 
\end{align*}
and so if $\kappa$ and $\beta$ are small enough we have 
\begin{align*}
|S_k \setminus S_k'' | \leq \frac{|S_k|}{4|D|^2}\leq \frac{|S_k^\sharp |}{2} .
\end{align*}
Choose an injection $f_k : S_k \setminus S_k'' \to S_k^\sharp$. 
Now since for each $k$ the set $Q_k = S_k^\sharp \setminus f_k (S_k \setminus S_k'')$ satisfies 
\begin{align*}
|Q_k | \geq |S_k^\sharp | - |S_k \setminus S_k''| \geq \frac{|S_k^\sharp |}{2}
\geq \frac{|S_k|}{4|D|^2} ,
\end{align*}
we deduce, in view of (\ref{E-prec}) and Remark~\ref{R-tower comparison}, that
\begin{align*}
X\setminus \bigsqcup_{k=1}^K S_k V_k \prec \bigsqcup_{k=1}^K Q_k V_k .
\end{align*}
We can thus find an open cover $\{ U_1 , \dots , U_r \}$ of 
$X\setminus \bigsqcup_{k=1}^K S_k V_k$ and $s_1 , \dots , s_r \in G$
such that $s_1 U_1 , \dots , s_r U_r$ are disjoint subsets of 
$\bigsqcup_{k=1}^K Q_k V_k$. Then the open sets 
$f_k (s)V_k = (f_k (s)s^{-1} )sV_k$ for $k=1,\dots ,K$ and $s\in S_k \setminus S_k''$ together with 
$s_1 U_1 , \dots , s_r U_r$ form a disjoint collection whose union is contained in 
$\bigsqcup_{k=1}^K S_k^\sharp V_k$, which shows that
\begin{align*}
X\setminus \bigsqcup_{k=1}^K S_k'' V_k \prec \bigsqcup_{k=1}^K S_k^\sharp V_k .
\end{align*}
Since the function $1-\varphi (1)$ is supported on $X\setminus\bigsqcup_{k=1}^K S_k'' V_k$
and $a$ takes the constant value $1$ on $\bigsqcup_{k=1}^K S_k^\sharp V_k$, it follows
by Lemma~\ref{L-partial} that there is a $v\in C(X)\rtimes_\lambda G$
satisfying $v^* a v = 1-\varphi (1)$. This shows that $1-\varphi (1) \precsim a$.
\end{proof}

\begin{example}\label{E-not af}
In \cite{GioKer10} examples were given of free minimal actions $\Zb\curvearrowright X$ 
on compact metrizable spaces such that the crossed product $C(X)\rtimes_\lambda G$ 
fails to be $\cZ$-stable.
As mentioned in Example~\ref{E-mdim}, these examples 
have finite tower dimension since they factor onto an odometer. By Theorem~\ref{T-af Z-stable},
they fail to be almost finite. 
\end{example}

Using Theorem~\ref{T-af Z-stable} we can give some new examples
of classifiable crossed products, as we now demonstrate. 
Let us write $\sC$ for the class of simple separable unital C$^*$-algebras having finite nuclear dimension 
and satisfying the UCT. This class is classified by the Elliott invariant (ordered $K$-theory paired with traces)
as a consequence of the work of Elliott--Gong--Lin--Niu \cite{EllGonLinNiu15}, Gong--Lin--Niu \cite{GonLinNiu15},
Tikuisis--White--Winter \cite{TikWhiWin17} in the stably finite case and of 
Kirchberg \cite{Kir94} and Phillips \cite{Phi00} in the purely infinite case.
Moreover, the stable finite C$^*$-algebras in the class $\sC$ are ASH algebras of topological dimension at most 2.
What is particularly novel in the examples below from the perspective of classification theory
is that one can combine infinite asymptotic dimension in the group
with positive topological entropy in the dynamics.
For a general reference on entropy
for actions of amenable groups see Chapter~9 of \cite{KerLi16}.

\begin{proposition}\label{P-res finite}
Suppose that $G$ is infinite, residually finite, and amenable. 
Let $r\in [0,\infty ]$.
Then there exists a uniquely ergodic free minimal action $G\curvearrowright X$ on the Cantor set
which is almost finite and has topological entropy $r$,
and the crossed product $C(X)\rtimes G$ belongs to the class $\sC$.
\end{proposition}

\begin{proof}
As $G$ is residually finite we can find a decreasing sequence 
$N_1 \supseteq N_2 \supseteq\dots$ of finite-index normal subgroups
of $G$ such that $\bigcap_{k=1}^\infty N_k = \{ e \}$. Then for each $k$ we
have the surjective homomorphism $G/N_{k+1} \to G/N_k$ given on cosets by $sN_{k+1} \mapsto sN_k$.
Form the inverse limit $Y$ of
the sequence $G/N_1 \leftarrow G/N_2\leftarrow\cdots$, which as a topological space 
is a Cantor set since $G$ is infinite. Then we have the free minimal action
$G\curvearrowright Y$ arising from the actions $(s,tN_k ) \mapsto stN_k$ of $G$ on each 
$G/N_k$. This is an example of a profinite action, which by definition is an 
inverse limit of actions on finite sets. 
It has a unique $G$-invariant Borel probability measure $\mu$,
namely the one induced from the uniform probability measures on the quotients $G/N_k$.
The p.m.p.\ action $G\curvearrowright (Y,\mu )$, being compact, has zero measure entropy.
By \cite{Wei01} there is a F{\o}lner sequence $\{ F_k \}$ for $G$ such that
for each $k$ the set $F_k$ is a complete set of representatives for the quotient of $G$ by $N_k$,
from which it follows that the action $G\curvearrowright Y$ is almost finite, since the clopen
partition of $Y$ corresponding to $N_k$ is the set 
of levels of a single clopen tower with shape $F_k$. 

By the Jewett--Krieger theorem \cite{Wei85,Ros87}
there is a uniquely ergodic minimal action $G\curvearrowright Z$ on the Cantor set 
whose invariant measure $\kappa$ gives a Bernoulli action with entropy $r$. 
Then the p.m.p.\ actions $G\curvearrowright (Y,\mu )$ and $G\curvearrowright (Z,\nu )$
are disjoint, which implies that the action $G\curvearrowright Y\times Z$ is uniquely ergodic and minimal,
as is readily seen. Moreover, the action $G\curvearrowright Y\times Z$ 
is free because the action $G\curvearrowright Y$ is free,
has entropy $r$ by the additivity of topological entropy and the variational principle, and
is almost finite by Theorem~\ref{T-af extension finite} since it factors onto the almost finite action 
$G\curvearrowright Y$ and the space $Z$ is zero-dimensional.

By Theorem~\ref{T-af Z-stable}, the crossed product $C(X)\rtimes_\lambda G$ is $\cZ$-stable.
Since the action is free, minimal, and uniquely ergodic, the crossed product
also has a unique tracial state, given by composing the canonical conditional expectation 
$C(X)\rtimes_\lambda G \to C(X)$ with integration with respect to the unique $G$-invariant
Borel probability measure on $X$. It follows by \cite{SatWhiWin15} that $C(X)\rtimes_\lambda G$
has finite nuclear dimension. Since $C(X)\rtimes_\lambda G$ satisfies the UCT \cite{Tu99},
we conclude that $C(X)\rtimes_\lambda G$ belongs to the class $\sC$.
\end{proof}


We note finally that profinite actions of countably infinite residually finite amenable groups,
such as the action $G\curvearrowright Y$ in the proof of Proposition~\ref{P-res finite},
are already known to be classifiable (see Remark~5.3 in \cite{SzaWuZac17}).

\section{The type semigroup and almost unperforation}\label{S-type}

Let $G\curvearrowright X$ be an action on a zero-dimensional compact Hausdorff space.
Write $\alpha$ for the induced action on $C(X)$, that is, 
$\alpha_s (f)(x) = f(s^{-1} x)$ for all $s\in G$, $f\in C(X)$, and $x\in X$.
On the space $C(X,\Zbn )$ of continuous functions
on $X$ with values in the set $\Zbn$ of nonnegative integers we define an equivalence relation
by declaring that $f\sim g$ if 
there are $h_1 ,\dots , h_n \in C(X,\Zbn )$ and $s_1 , \dots , s_n \in G$ such that
$\sum_{i=1}^n h_i = f$ and $\sum_{i=1}^n \alpha_{s_i} (h_i) = g$
(transitivity is not immediately obvious but is readily checked).
Write $S(X,G)$ for the quotient $C(X,\Zbn )/\sim$. This is an Abelian semigroup under the 
operation $[f]+[g] = [f+g]$, which is easily seen to be well defined. 
We moreover endow $S(X,G)$
with the algebraic order, that is, for $a,b\in S(X,G)$ we declare that $a\leq b$ whenever 
there exists a $c\in S(X,G)$ such that $a + c = b$.
The ordered Abelian semigroup $S(X,G)$ is called the {\it (clopen) type semigroup} of the action.
Note that, in view of Proposition~\ref{P-clopen subequivalence}, comparison can be expressed in this language 
by saying that, for all clopen sets $A,B\subseteq X$, if $\mu (A) < \mu (B)$ for all $\mu\in M_G (X)$ 
then $[\unit_A ] \leq [\unit_B ]$.

One can equivalently define $S(X,G)$ by considering the collection of clopen subsets of $X\times\Nb$
of the form $\bigsqcup_{i=1}^n A_i \times \{ i\}$ for some $n\in\Nb$ 
(the {\it bounded} subsets of $X\times\Nb$) and quotienting
by the relation of equidecomposability, under which 
$\bigsqcup_{i=1}^n A_i \times \{ i\} \sim \bigsqcup_{i=1}^m B_i \times \{ i\}$
if for each $i=1,\dots , n$ there exist a clopen partition $\{ A_{i,1} , \dots , A_{i,J_i} \}$ of $A_i$, 
$s_{i,1} , \dots , s_{i,J_i} \in G$, and $k_{i,1} ,\dots , k_{i,J_i} \in \{ 1,\dots , m \}$ such that 
\[
\bigsqcup_{i=1}^n \bigsqcup_{j=1}^{J_i} s_{i,j} A_{i,j} \times \{ k_{i,j} \}
= \bigsqcup_{i=1}^m B_i \times \{ i\} .
\]
The semigroup operation is the concatenation
\begin{align*}
\bigg[ \bigsqcup_{i=1}^n A_i \times \{ i\} \bigg] 
+ \bigg[ \bigsqcup_{i=1}^m B_i \times \{ i\} \bigg] 
= \bigg[ \bigg(\bigsqcup_{i=1}^n A_i \times \{ i\} \bigg) 
\sqcup \bigg( \bigsqcup_{i=n+1}^{n+m} B_i \times \{ i\} \bigg) \bigg] .
\end{align*}
The isomorphism with the first construction can be seen by considering for each function $f\in C(X,\Zbn )$
the decomposition $f = \sum_{i=1}^n \unit_{A_i}$ where 
$A_i = \{ x\in X : f(x)\geq i \}$ and $n = \max_{x\in X} f(x)$ and associating to $f$
the subset $\bigsqcup_{i=1}^n A_i \times \{ i \}$ of $X\times \{ 1,\dots , n\}$.

The idea of a type semigroup originates in Tarski's work on amenability (see \cite{Wag93}) 
and has variants depending on the type of action (e.g., on an ordinary set,
on a measure space, or on a zero-dimensional compact space) and the types of sets 
use in the definition equidecomposability (e.g., arbitrary, measurable, or clopen).
The clopen version was studied in \cite{RorSie12}.

Note that every measure $\mu$ in $M_G (X)$ induces a state on $S(X,G)$ given by
$[f] \mapsto \mu (f)$. When the action is minimal, this gives a bijection from measures 
in $M_G (X)$ to states on $S(X,G)$ by the proof of Lemma~5.1 in \cite{RorSie12}.

An ordered Abelian semigroup $A$ is said to be {\it almost unperforated} if, 
for all $a,b\in A$ and $n\in\Nb$, the inequality $(n+1)a \leq nb$ implies $a\leq b$.

\begin{lemma}\label{L-comparison ap}
Let $G\curvearrowright X$ be a free minimal action on the Cantor set 
such that $S(X,G)$ is almost unperforated. Then the action has comparison.
\end{lemma}

\begin{proof}
Let $A$ and $B$ be clopen subsets of $X$ such that $\mu (A) < \mu (B)$ for all $\mu\in M_G (X)$.
By Lemma~5.1 of \cite{RorSie12}, every state $\sigma$ on $S(X,G)$ gives rise to a measure 
$\mu$ in $M_G (X)$ satisfying $\mu (A) = \sigma ([\unit_A ])$ for every clopen set $A\subseteq X$.
Thus $\sigma ([\unit_A ]) < \sigma ([\unit_B ])$ for every state $\sigma$ on $S(X,G)$.
Since the action is minimal, $X$ is covered by finitely many translates of $B$,
so that $[\unit_A ] \leq m[\unit_B ]$ for some $m\in\Nb$. It follows 
by Proposition 2.1 of \cite{OrtPerRor12} that there exists
an $n\in\Nb$ for which $(n+1)[\unit_A ]\leq n[\unit_B ]$. Almost unperforation
then yields $[\unit_A]\leq [\unit_B]$, establishing comparison.
\end{proof}

The following is a generalization of Theorem~\ref{T-comparison af}(i)$\Rightarrow$(ii).
It shows that, for free minimal actions on the Cantor set, 
almost finiteness implies a stable version of comparison.

\begin{lemma}\label{L-type af}
Let $G\curvearrowright X$ be a free minimal action on the Cantor set which is almost finite.
Let $f,g\in C(X,\Zbn )$ be such that $\mu (f) < \mu (g)$
for every $\mu\in M_G (X)$. Then $[f]\leq [g]$.
\end{lemma}

\begin{proof}
Write $f = \sum_{j=1}^n \unit_{A_j}$ and $g = \sum_{k=1}^m \unit_{B_k}$
where $A_j = \{ x\in X : f(x)\geq j \}$ and $B_j = \{ x\in X : g(x)\geq k \}$,
with $n = \max_{x\in X} f(x)$ and $m = \max_{x\in X} g(x)$.

Let $K$ be a finite subset of $G$ and $\delta > 0$, both to be determined.
By Theorem~\ref{T-zero dim}, there is a clopen castle $\{ (V_i , S_i ) \}_{i\in I}$
with $(K,\delta )$-invariant shapes such that $\bigsqcup_{i\in I} S_i V_i = X$.
By replacing each tower $(V_i ,S_i )$ with finitely many thinner towers with the same shape $S_i$
and with bases forming a clopen partition of $V_i$, we may assume that
every level of every tower in the castle is either contained in or disjoint from $A_j$
for every $j=1,\dots , n$, and also either contained in or disjoint from $B_k$
for every $k=1,\dots ,m$. 
For each $i$ set $E_{i,j} = \{ s\in S_i : sV_i \subseteq A_j \}$ for every $j$
and $F_{i,k} = \{ s\in S_i : sV_i \subseteq B_k \}$ for every $k$.
An argument by contradiction using a weak$^*$ cluster point as in the proof
of Theorem~\ref{T-tower} shows that
we can choose $K$ and $\delta$ so that for every $i$ we have
\begin{align*}
\frac{1}{|S_i|} \sum_{j=1}^n |E_{i,j}| 
\leq \frac{1}{|S_i|} \sum_{k=1}^m |F_{i,k}| ,
\end{align*}
in which case we can find a bijection 
$\bigsqcup_{j=1}^n E_{i,j}\times \{ j\} \to \bigsqcup_{k=1}^m F_{i,k}\times \{ k\}$,
which we write as $(s,j) \mapsto (t_{i,s,j} , k_{i,s,j})$.
We then have
\begin{align*}
f = \sum_{i\in I} \sum_{j=1}^n \sum_{s\in E_{i,j}} \unit_{sV_i} 
\end{align*}
and
\begin{align*}
\sum_{i\in I} \sum_{j=1}^n \sum_{s\in E_{i,j}} \alpha_{t_{i,s,j} s^{-1}} (\unit_{sV_i} ) 
= \sum_{i\in I} \sum_{j=1}^n \sum_{s\in E_{i,j}} \unit_{t_{i,s,j} V_i} 
\leq g ,
\end{align*}
showing that $[f]\leq [g]$, as desired.
\end{proof}

The following adds almost unperforation to the conditions in Theorem~\ref{T-comparison af}
in the case that the space is the Cantor set.

\begin{theorem}\label{T-Cantor comparison af}
Let $G\curvearrowright X$ be a free minimal action on the Cantor set
and consider the following conditions:
\begin{enumerate}
\item the action is almost finite,

\item $S(X,G)$ is almost unperforated,

\item the action has comparison,

\item the action has $m$-comparison for all $m\in\Nb$,

\item the action has $m$-comparison for some $m\in\Nb$.
\end{enumerate}
Then (i)$\Rightarrow$(ii)$\Rightarrow$(iii)$\Rightarrow$(iv)$\Rightarrow$(v),
and if $E_G (X)$ is finite then all five conditions are equivalent.
\end{theorem}

\begin{proof}
(i)$\Rightarrow$(ii). Let $f,g\in C(X,\Zbn )$ be such that $(n+1)[f] < n[g]$ for some $n\in\Nb$.
Then for every $\mu\in M_G (X)$ we have $(n+1)\mu (f) < n\mu (g)$ and hence $\mu (f) < \mu (g)$.
It follows by Lemma~\ref{L-type af} that $[f]\leq [g]$, establishing almost unperforation.

(ii)$\Rightarrow$(iii). This is Lemma~\ref{L-comparison ap}.

(iii)$\Rightarrow$(iv)$\Rightarrow$(v). Trivial.

(v)$\Rightarrow$(i). This is a special case of Theorem~\ref{T-comparison af}(iv)$\Rightarrow$(i).
\end{proof}

\end{document}